\documentclass[10pt,reqno]{amsart}

\usepackage{amssymb}
\usepackage{amsfonts}
\usepackage{amsthm}
\usepackage{amsmath}

\numberwithin{equation}{section}
\allowdisplaybreaks

\newtheorem{theorem}{Theorem}[section]
\newtheorem{proposition}[theorem]{Proposition}

\newtheorem{lemma}[theorem]{Lemma}
\newtheorem{corollary}[theorem]{Corollary}

\theoremstyle{definition}
\newtheorem{definition}[theorem]{Definition}

\theoremstyle{remark}

\newcommand{\R}{\mathbb{R}}

\newcommand{\Z}{\mathbb{Z}}

\newcommand{\eps}{\varepsilon}

\newcommand{\scriptE}{\mathcal{E}}
\newcommand{\scriptF}{\mathcal{F}}
\newcommand{\scriptJ}{\mathcal{J}}
\newcommand{\scriptK}{\mathcal{K}}
\newcommand{\scriptL}{\mathcal{L}}

\newcommand{\tnorm}{\textnormal}

\newcommand{\qtq}[1]{\quad\text{#1}\quad}
\newcommand{\ctc}[1]{\,\,\text{#1}\,\,}

\DeclareMathOperator*{\dist}{dist}

\begin{document}

\title[Uniform bounds for averages along degenerate curves]{Uniform bounds for convolution and restricted X-ray transforms along degenerate curves}
\author{Spyridon Dendrinos}
\address{School of Mathematical Sciences, University College Cork, Cork, Ireland} 
\email{sd@ucc.ie}
\author{Betsy Stovall}
\address{Department of Mathematics, University of Wisconsin, Madison, WI 53726}
\email{stovall@math.wisc.edu}

\begin{abstract}
We establish endpoint Lebesgue space bounds for convolution and restricted X-ray transforms along curves satisfying fairly minimal differentiability hypotheses, with affine and Euclidean arclengths.  We also explore the behavior of certain natural interpolants and extrapolants of the affine and Euclidean versions of these operators.
\end{abstract}

\maketitle

\section{Introduction}

This article deals with the basic problem of determining the precise amount of $L^p$-improving for certain weighted averaging operators associated to curves in $\R^d$. In the unweighted case, this problem has been studied by Tao and Wright in wide generality, and in \cite{TW}, they completely describe (except for boundary points) the set of $(p,q)$ for which these operators map $L^p$ boundedly into $L^q$, under certain smoothness hypotheses and in the presence of a cutoff.  This set of $(p,q)$ depends on the torsion (and appropriate generalizations thereof), but if instead the average is taken against an `affine arclength measure,' the effects of vanishing torsion are mitigated, and the $(p,q)$ region is larger.  In fact (excepting boundary points), the new region is essentially independent of the curves \cite{BSmult}.  

We are interested in the questions of whether the endpoint estimates hold, whether there is a natural way to relate the weighted and unweighted versions of these operators, and to what extent the regularity hypotheses in previous articles (often $C^\infty$) can be relaxed.  

Endpoint bounds have been established in a number of special cases.  A more extensive list of references is given in \cite{DLW, DSjfa}; we will focus here on the most recent results.   In \cite{GressIMRN}, Gressman proved that in the polynomial case of the Tao--Wright theorem, endpoint restricted weak type estimates hold, but left open the question of strong type bounds.  In the translation-invariant case, more tools are available, and correspondingly, more is known.  In \cite{DLW, oberlin-polynomial, BSjfa}, endpoint strong type estimates for convolution with affine arclength measure on polynomial curves were proved.  These estimates depend only on the dimension and polynomial degree and require no cutoff function.  In \cite{DSjfa}, an analogous result was proved for the restricted X-ray transform.  

For the low regularity case, much less is known.  We are primarily motivated by the recent \cite{OberlinJFA} and \cite{DMtams}.  In \cite{OberlinJFA}, Oberlin proved bounds along the sharp line for convolution with affine arclength measure along low-dimensional `simple' curves satisfying certain monotonicity and log-concavity hypotheses.  In particular, there exist infinitely flat curves satisfying these hypotheses.  This provides further motivation for the consideration of affine arclength measure, because in these cases there are simply no nontrivial estimates for the unweighted operators.  In \cite{DMtams}, the first author and M\"uller proved that restriction to certain $C^d$ perturbations of monomial curves with affine arclength measure satisfies the same range of $L^p \to L^q$ inequalities as restriction to nondegenerate curves.  

Our purpose here is to generalize, to the extent possible, the endpoint results mentioned above to more general classes of curves of low regularity. To address the question of the natural relationship between the weighted and unweighted operators, we show how, by a simple interpolation argument, `weaker' estimates for operators with `larger' weights (including the optimal estimates in the unweighted case) can be deduced from the optimal estimates for the affine arclength case.  Finally, motivated by similarities between restriction operators and generalized Radon transforms, we prove an analogue of the main result of \cite{DMtams} for convolution with affine arclength measure along monomial-like curves with only $d$ derivatives.

\section{Results and methods}

Let $I \subset \R$ be an interval and $\gamma \in C^d_{\rm{loc}}(I;\R^d)$; that is, $\gamma:I \to \R^d$ is a curve in $C^d(K)$ for every compact sub-interval $K \subseteq I$.  We define the torsion $L_\gamma$ and affine arclength measure $\lambda_\gamma\, dt$ by
$$
L_\gamma = \det(\gamma',\ldots,\gamma^{(d)}), \qquad \lambda_\gamma = |L_\gamma|^{\frac2{d(d+1)}}.
$$
Since $\lambda_{\gamma\circ\phi} = |\phi'| \lambda_\gamma \circ \phi$, $\lambda_\gamma\, dt$ is naturally interpreted as a measure on the image of $\gamma$.  The behavior of affine arclength measure under affine transformations of $\gamma$ (especially in contrast to Euclidean arclength measure) make it particularly well-suited to harmonic analysis.

We are primarily interested in the convolution operator
$$
T_\gamma f(x) = \int_I f(x-\gamma(t))\, \lambda_\gamma(t)\,dt, \qquad x \in \R^d,
$$
and the X-ray transform, 
$$
X f(x,y) = \int_{\R} f(s,x+sy) \, ds, \qquad (x,y) \in \R^{d+d}.
$$
The latter averages $f$ along each line parallel to $(1,y)$, and we restrict $y$ to lie along the image of $\gamma$:
$$
X_\gamma(t,x) := X(x,\gamma(t)).  
$$

It is known that, aside from the trivial case $L_\gamma \equiv 0$, the natural endpoint $L^p \to L^q$ bounds for these operators are
\begin{align}\label{E:convl endpt}
\|T_\gamma f\|_{L^{q_d}(\R^d)} \lesssim \|f\|_{L^{p_d}(\R^d)}, \qquad &(p_d,q_d) = (\tfrac{d+1}2,\tfrac{d(d+1)}{2(d-1)})\\ \label{E:X ray endpt}
\|X_\gamma f\|_{L^{s_d}(\R^{1+d};dx\,\lambda_\gamma dt)} \lesssim \|f\|_{L^{r_d}(\R^{1+d})}, \qquad &(r_d,s_d) = (\tfrac{(d+1)(d+2)}{d^2+d+2},\tfrac{d+2}d),
\end{align}
and the interest is in obtaining these bounds with implicit constants that are uniform over some large class of curves.  Such estimates have been established for polynomial curves of bounded degree \cite{DLW, DSjfa, oberlin-polynomial, BSjfa}; our primary goal is to relax that regularity assumption to the extent possible.  

It is clear, however, that some further restrictions must be made.  Indeed, simple examples show that for \eqref{E:convl endpt} or \eqref{E:X ray endpt} to hold, we must have
\begin{equation} \label{E:cvx hull}
\int_{I'}\lambda_\gamma(t)\, dt \lesssim |\rm{ch}(\gamma(I'))|^{\frac2{d(d+1)}}, \quad \text{for all intervals $I' \subseteq I$},
\end{equation}
where `$\rm{ch}$' indicates the convex hull.  For sufficiently small intervals $I'$ and finite type curves $\gamma$, both sides of this inequality are comparable, as discussed in \cite{Oberlin03}, but without these hypotheses on the interval and curve, this may fail.  Consider the examples
\begin{align*}
\gamma^1(t) &= (t,\sin(t^{-k}) \exp(-t^{-2})), \quad 0 < t \leq 1\\
\gamma^2(t) &= ((1+\exp(-t))\cos t , (1+\exp(-t))\sin t), \qquad t > 0\\
\gamma^3(t) &= (t, \sin t, \cos t), \qquad t \in \R,
\end{align*}
where $k$ is taken to be sufficiently large in the case of $\gamma^1$.  (This first example is due to Sj\"olin \cite{Sjolin}.)  Even though each $\gamma^j$ is an injective immersion, and $\gamma^2$ and $\gamma^3$ have nonvanishing torsion, none of these satisfy \eqref{E:cvx hull} globally.   

\subsection{Log-concave torsions}  

The examples $\gamma^2$ and $\gamma^3$ show that it is necessary to control the oscillation of lower dimensional projections of $\gamma$, and not just $L_\gamma$.  We define 
$$
L_{\gamma}^j = L_{(\gamma_1,\ldots,\gamma_j)}
$$
and
$$
B_\gamma^k = L_{\gamma} (L_{\gamma}^{d-k-1})^k (L_{\gamma}^{d-k})^{-(k+1)}, \qquad 1 \leq k \leq d; 
$$
here we are using the convention that $L_\gamma^0 = L_\gamma^{-1} = 1$, so, for example, $B_\gamma^d = L_\gamma$.  

Our most general results are restricted weak type analogues of \eqref{E:convl endpt} and \eqref{E:X ray endpt}, which we prove in Section~\ref{S:RWT}.  For simplicity, we summarize these in somewhat less than their full generality; the more general hypotheses in Section~\ref{S:RWT} will be analogous to those in \cite{OberlinJFA}.  We say that a function $f:I \to [0,\infty)$ is log-concave if $f(\theta t_1+(1-\theta)t_2) \geq f(t_1)^\theta f(t_2)^{1-\theta}$.  

\begin{theorem} \label{T:Intro RWT}
Let $\gamma\in C^d_{\rm{loc}}(I;\R^d)$.  Assume that the $B_\gamma^k$ are log-concave for $1 \leq k \leq d$.  Then we have the restricted weak type estimates
\begin{gather}\label{E:Intro RWT}
\langle T_\gamma \chi_E,\chi_F \rangle \leq C_d |E|^{\frac1{p_d}} |F|^{\frac1{q_d'}}, \qquad (p_d,q_d) = (\tfrac{d+1}2,\tfrac{d(d+1)}{2(d-1)})\\ \label{E:Intro X RWT}
\langle \lambda_\gamma(t)^{\frac1{s_d}} X_\gamma \chi_G,\chi_H \rangle \leq C_d |G|^{\frac1{r_d}}|H|^{\frac1{s_d'}}, \qquad (r_d,s_d) = (\tfrac{(d+1)(d+2)}{d^2+d+2},\tfrac{d+2}d),
\end{gather}
for all positive measure Borel sets $E,F \subseteq \R^d$, $G,H \subseteq \R^{1+d}$.  
\end{theorem}

Because $T_\gamma^*\chi_F(y) = T_\gamma \chi_{(-F)}(-x)$ this implies that $T_\gamma$ is also of restricted weak type $(q_d',p_d')$, so $T_\gamma$ is of strong type $(p,q)$ for all interpolants of $(p_d,q_d)$ and $(q_d',p_d')$.  

The proof consists of two parts.  In Section~\ref{S:RWT}, we prove Propositions~\ref{P:RWT} and~\ref{P:RWTXray}, which roughly state that if $L_\gamma = B_\gamma^d$ is log-concave and the geometric inequality 
\begin{equation} \label{E:Intro gi}
|\det(\gamma'(t_1),\ldots,\gamma'(t_d))| \geq c \prod_{j=1}^d |L_\gamma(t_j)|^{\frac1d}\prod_{i < j}|t_i-t_j|, \qquad (t_1,\ldots,t_d) \in I^d
\end{equation}
holds, then the restricted weak type estimates \eqref{E:Intro RWT} and \eqref{E:Intro X RWT} hold (with constants depending on $d$ and $c$).  We note that the left hand side of \eqref{E:Intro gi} is the Jacobian of $\sum_{j=1}^d \gamma(t_j)$, so this may be thought of as a stronger, more quantitative version of \eqref{E:cvx hull}.  

For the proof of the restricted weak type estimates, we use the method of refinements, but with a twist in the case of Proposition~\ref{P:RWT}:  by ordering certain parameters (this idea was suggested by Phil Gressman, personal communication), we avoid the complicated band structure argument of \cite{ChCCC}.  

The other half of the argument is the proof of the geometric inequality.  The following is a simplified but weaker version of Proposition~\ref{P:convl gi}, which is proved in Section~\ref{S:gi proofs}.  

\begin{proposition} \label{P:Intro convl gi}
Let $\gamma\in C^d_{\rm{loc}}(I;\R^d)$.  If $B_\gamma^k$ is log-concave for each $1 \leq k \leq d-1$, we may decompose $I = \bigcup_{j=1}^{C_d} I_j$, where each $I_j$ is an interval and $(t_1,\ldots,t_d) \in I_j^d$ implies that \eqref{E:Intro gi} holds, with a constant depending only on $d$.
\end{proposition}

We note that in two dimensions, if $\gamma \in C^3$, $\gamma_1' \neq 0$, and $L_\gamma \neq 0$, we may reparametrize $\gamma$ so that $B_\gamma^1 \equiv 1$, and hence \eqref{E:Intro gi} holds for the reparametrization.  In higher dimensions, it seems harder to determine whether a given curve has a parametrization that satisfies the geometric inequality.  

As an example, consider the monomial curve $\gamma(t) = (t^{a_1},\ldots,t^{a_d})$ with $a_i \in \R$, $a_1 < \cdots < a_d$.  Then 
$$
B_\gamma^k = c_a t^{a_{d-k+1}+\cdots+a_d - ka_{d-k}-\frac{k(k+1)}2}
$$
(we interpret $a_0$ as 0), which is log-concave if and only if the exponent is nonnegative.  In particular the hypothesis of Theorem~\ref{T:Intro RWT} holds for a reparametrization of $\gamma$ (by either $t\mapsto t^N$ or $t\mapsto t^{-N}$ for some large $N$).  In this way, we obtain endpoint restricted weak type estimates that are completely independent of the $a_i$.  By contrast, the geometric inequality \eqref{E:Intro gi} is not parametrization invariant.  By considering simple configurations (take one $t_i$ very small and the others moderate), we see that log-concavity of $B_\gamma^1$ is actually necessary for the geometric inequality to hold.  In two dimensions, this is both necessary and sufficient (by Proposition~\ref{P:Intro convl gi}).  In higher dimensions, it is clear that further inequalities relating the exponents should be necessary, but it is possible that a slightly weaker condition than log-concavity of the $B_\gamma^k$ suffices in the monomial case.  We mention this in part because it seems to be suggested in \cite{GressMonomial} (though no argument is given in the case of non-integer powers) that a geometric inequality essentially of the form \eqref{E:Intro gi} holds for monomial curves with sufficiently large but otherwise arbitrary real powers.  This is not the case.  

There is a close connection between generalized Radon transforms and Fourier restriction operators, and it would be interesting to see whether improved restriction estimates could be obtained for curves satisfying the hypotheses of Theorem~\ref{T:Intro RWT}.  In the case of simple curves $\gamma(t) = (t,\ldots,t^{d-1},\phi(t))$, such estimates were obtained in \cite{BOS2}.  In the general case, Proposition~\ref{P:Intro convl gi} and its proof may be useful, but current technology for Fourier restriction estimates requires additional geometric inequalities (see \cite{DMtams, DW, DrM2}).


\subsection{Interpolation}
In Section~\ref{S:interpolation}, we give a simple interpolation argument that can be used in conjunction with the restricted weak type results of Section~\ref{S:RWT}.  

In the special case $|L_\gamma(t)| \sim t^k$, for some real number $k \neq -\frac{d^2+d}2$, these give fractional integral analogues of \eqref{E:Intro RWT}.  In particular, when $k \geq 0$, we can recover the unweighted endpoint restricted weak type estimates (and hence strong type bounds on the interior of a line segment for the convolution operator).  We will obtain the corresponding strong type bounds in Section~\ref{S:ST}, but (as will be seen) interpolation yields a simpler argument and gives better bounds in the interior.  

In more general cases, some of the operators and estimates that arise in this way seem a little surprising (or did to the authors).

\subsection{Strong type bounds} \label{SS:ST}
In Section \ref{S:ST}, we turn our attention to the endpoint strong type bounds for the convolution operator $T_\gamma$ and related operators with different weights.  We are not yet able to address these at the same level of generality as in Theorem~\ref{T:Intro RWT}, so in the spirit of \cite{DMtams,DrM2} we consider the following family of monomial-like curves, which contains all curves of finite type.

Let $a_1 < a_2 < \cdots < a_d$ be nonzero (but possibly negative) real numbers and consider the monomial-like curve $\gamma(t) = (t^{a_1}\theta_1(t),\ldots,t^{a_d}\theta_d(t))$, $0 < t < T$.  Here $\theta_i \in C^d_{\rm{loc}}((0,T))$ and satisfies
\begin{equation} \label{E:thetas}
\lim_{t \searrow 0} \theta_i(t) = \theta_i(0) \in \R\setminus\{0\}, \qquad \lim_{t \searrow 0} t^m \theta_i^{(m)} (t) = 0, \quad 1 \leq i,m \leq d.
\end{equation}
We prove the following partial analogue of the main theorem in \cite{DMtams}.  

\begin{theorem}\label{T:convl st}
For each monomial-like curve $\gamma$ as above, there exists $\delta=\delta_\gamma > 0$ such that for each $0 < \theta \leq 1$, the operator
$$
T^\theta_\gamma f(x) = \int_0^\delta f(x-\gamma(t))\, \lambda_\gamma(t)^{\theta} \, \tfrac{dt}{t^{1-\theta}}
$$
satisfies
\begin{equation} \label{E:convl st}
\|T^\theta_\gamma f\|_{L^{q_d/\theta}(\R^d)}\lesssim  \|f\|_{L^{p_d/\theta}(\R^d)}, \qquad (p_d,q_d) = \Big(\tfrac{d+1}{2},\tfrac{d(d+1)}{2(d-1)}\Big),
\end{equation}
for all compactly supported Borel functions $f$.  If $\theta=1$, the implicit constant in \eqref{E:convl st} may be chosen to depend only on $d$ and an upper bound for $\tfrac{|a_1+\cdots+a_d|}{a_d-a_1}$, and if $0 < \theta < 1$, the implicit constant may be taken to depend on $d$, $\theta$, $|a_1+\cdots+a_d|$, and $a_d-a_1$.  
\end{theorem}

Since the operator $T^\theta_\gamma$ is essentially self-adjoint, it follows that \eqref{E:convl st} also holds with $(\tfrac{p_d}{\theta},\tfrac{q_d}{\theta})$ replaced by the dual pair $((\tfrac{q_d}{\theta})',(\tfrac{p_d}{\theta})')$, with the same bounds, so it also holds for all interpolants of these pairs.  We note that the limiting operator as $\theta \searrow 0$ is the one-sided (and hence very unbounded) Hilbert transform.  

There is a related result in \cite{GressMonomial}, but we obtain strong-type (rather than restricted weak type) bounds up to the natural affine arclength (rather than unweighted) endpoint for a larger class of curves.  

It should be possible to obtain an analogous theorem for the restricted X-ray transform via similar techniques, but the authors have not undertaken to verify this.

For the proof, motivated by \cite{DMtams}, we use the exponential parametrization $t \mapsto e^{-t}$; this also avoids the issues discussed above about the rate of increase of the $a_i$ and geometric inequalities.  Curiously, in this parametrization, the standard geometric inequality \eqref{E:Intro gi} is insufficient to obtain strong type bounds via our methods, so in Section~\ref{S:gi proofs}, we establish an exponential improvement, Proposition~\ref{P:convl exp gi}.  We prove Theorem~\ref{T:convl st} in Section~\ref{S:ST} using the method of refinements.  Again we order certain parameters to avoid the band structure argument.  Though the ordering is a bit more difficult for the strong type bounds than restricted weak type bounds, it still results in a substantially shorter argument than appeared in related articles (cf.\ \cite{BSlms, BSjfa}).  

A potentially interesting line of further questioning would be the precise dependence of the implicit constant in \eqref{E:convl st} on the $a_i$.  If $A:=\sum a_i = 0$, applying Theorem~\ref{T:convl st} to a reparametrization by $t\mapsto t^{1/(a_d-a_1)}$ implies that \eqref{E:convl st} actually holds with implicit constant equal to $C_{d,\theta}(a_d-a_1)^{\theta-1}$.  This seems to be essentially optimal because it scales in the right way under reparametrizations $t \mapsto t^k$.  

If $A \neq 0$, applying the methods of Section~\ref{S:RWT} to the exponential parametrization, we see that the restricted weak type version of \eqref{E:convl st} holds with a constant depending only on $d$, in the case $\theta=1$.  Using this and applying the interpolation in Section~\ref{S:RWT} to the reparametrization by $t \mapsto t^{1/|A|}$, the restricted weak type version of \eqref{E:convl st} holds with implicit constant $C_{d,\theta}|A|^{\theta-1}$, which again seems essentially optimal.  

In the case of strong type bounds, however, the authors have not been able to remove the dependence on $\tfrac{|A|}{a_d-a_1}$ or to show that it is necessary, even in the very simple case $\gamma_n(t) = (t^n,t^{n+1})$.\footnote{We would like to point out that the ratio $|A|/(a_d-a_1)$ is invariant under the above power-type reparametrizations.  We also note that this dependence does not seem to simply be an artifact of the exponential parametrization, and also seems to arise when one uses the methods of previous articles on the subject, such as \cite{DLW,oberlin-polynomial,BSjfa}.}  This sequence of examples is closely related to the curve $\gamma(t) = (e^{-\frac1t},te^{-\frac1t})$, for which the authors have not been able to prove or disprove strong type bounds.  It would be somewhat surprising if endpoint strong type bounds do not hold for the latter curve because, unlike all known counter-examples, it displays no oscillatory behavior.  

\subsection*{Acknowledgements}  
The first author was supported by Deutsche Forschungsgemeinschaft grant DE1517/2-1.  The second author was partially supported by NSF DMS 0902667 and 1266336.  

\subsection*{Notation}  
If $A$ and $B$ are two nonnegative quantities, we will write $A \lesssim B$ if $A \leq CB$ for some innocuous constant $C$, which will be allowed to change from line to line.  The meaning of `innocuous' will be allowed to change from proof to proof, but will always be clear from the context (or explicitly given).  By $A\sim B$ we mean $A \lesssim B \lesssim A$. 

\section{Restricted weak type bounds}
\label{S:RWT}

We now state our restricted weak type results in their full generality.  We are interested in curves $\gamma$ for which the torsion and Jacobian determinant $J_\gamma$ of $\Phi_\gamma(t_1,\ldots,t_d) := \sum_{j=1}^d (-1)^j \gamma(t_j)$ are related by the geometric inequality
\begin{equation} \label{E:convl GI}
|J_\gamma(t_1,\ldots,t_d)| \geq C_\gamma \prod_{j=1}^d |L_\gamma(t_j)|^\frac1d \prod_{1 \leq i < j \leq d}|t_j-t_i|.
\end{equation}
Our main results are conditioned on the validity of this geometric inequality, together with a log-concavity assumption on $\lambda_\gamma$.  

\begin{proposition}\label{P:RWT}
Let $I$ be an open interval and $\gamma:I \to \R^d$ a $d$-times continuously differentiable curve.  Assume that $L_\gamma$ satisfies the geometric inequality \eqref{E:convl GI} on $I^d$.  Assume further that $\lambda_\gamma(t) \sim f(t)$ for some log-concave function $f$ on $I$.  Then for all Borel sets $E,F \subseteq \R^d$ having finite measures,
\begin{equation} \label{E:RWT}
\langle T_\gamma \chi_E,\chi_F \rangle \lesssim |E|^{\frac1{p_d}}|F|^{\frac1{q_d'}}, \qquad (p_d,q_d') = \Big(\tfrac{d+1}{2},\tfrac{d(d+1)}{d^2-d+2}\Big).
\end{equation}
The implicit constant in \eqref{E:RWT} depends only on $d$, the constant $C_\gamma$ in \eqref{E:convl GI}, and the implicit constants bounding $\lambda_\gamma$ in terms of $f$.  
\end{proposition}

\begin{proposition}\label{P:RWTXray}
Suppose $\gamma: I \to \R^d$ satisfies the same hypotheses as in Proposition~\ref{P:RWT}.  Then for all finite measure Borel sets $E,F \subseteq \R^{1+d}$,
\begin{equation} \label{E:RWTXray}
\langle \lambda_\gamma(t)^{\frac1{q_d}} X_\gamma \chi_E,\chi_F \rangle \lesssim |E|^{\frac1{p_d}}|F|^{\frac1{q_d'}}, \qquad (p_d,q_d') = \Big(\tfrac{(d+1)(d+2)}{d^2+d+2},\tfrac{d+2}{2}\Big).
\end{equation}
The implicit constant in \eqref{E:RWTXray} has the same dependence as that in \eqref{E:RWT}.  
\end{proposition}

This raises the question of when the geometric inequality \eqref{E:convl GI} holds.  It is known in the case when $\gamma$ is a polynomial (\cite{DW}), when $\gamma$ is monomial-like (\cite{DMtams}), and for simple curves (i.e.\ those of the form $(t,\ldots,t^{d-1},\phi^{(d)}(t)))$ satisfying certain monotonicity and almost log-concavity hypotheses (\cite{BOS2, OberlinJFA}).  We will show that it holds for much more general curves, under hypotheses analogous to the latter case.  To state our result, we need a couple of definitions.

\begin{definition} \label{D: almost monotone}
Given a positive constant $C$, we call a non-negative function $f:I \subseteq \R \to\R$ {\it $C$-almost increasing} if $f(t_1) \leq C f(t_2)$ whenever $t_1 \leq t_2$.  We define \textit{$C$-almost decreasing} analogously and also say that a function is {\it $C$-almost monotone} if it is either $C$-almost increasing or $C$-almost decreasing.
\end{definition}

Following Oberlin \cite{OberlinJFA}, we also make the following definition.

\begin{definition} \label{D: almost log-concave}
Given a positive constant $M$, we call a function $f:\R\to\R$ {\it M-almost log-concave} if
\[
M \Big|f\Big(\frac{t_1+t_2}{2}\Big)\Big|\ge  |f(t_1)|^{1/2} |f(t_2)|^{1/2},
\]
for all $t_1, t_2 \in I$.
\end{definition}

The following proposition, which will be used in the proof of Propositions \ref{P:RWT} and \ref{P:RWTXray}, will be proved in Section~\ref{S:gi proofs}. Such geometric inequalities have been key ingredients in most of the recent proofs of endpoint estimates \cite{DSjfa, OberlinJFA, BSjfa}, because a good lower bound for the Jacobian of the mapping $\Phi$, which arises from an iteration procedure, is a central feature of the method of refinements, which originated in Christ \cite{ChCCC}. These geometric inequalities were first proved in the context of Fourier restriction to curves (see \cite{DMtams, DW, DrM2}). The curves $\gamma$ considered here are direct generalizations of the curves considered by Oberlin \cite{OberlinJFA}. However, our proof resembles more the one in \cite{DMtams}.

\begin{proposition} \label{P:convl gi}
Let $I$ be an open interval and $\gamma:I \to \R^d$ a $d$-times continuously differentiable curve.  Assume that the $L_\gamma^j$, $1 \leq j \leq d$, never vanish on $I$ and that the functions 
$$
L_{\gamma} (L_\gamma^{d-k-1})^k (L_\gamma^{d-k})^{-k-1}, \quad 1\le k \le d-1,
$$
are $C$-almost monotone and $M$-almost log-concave for some $C,M>0$.  Then, for all $(t_1,\ldots,t_d) \in I^d$,
\begin{equation} \label{E:convl gi}
|J_\gamma(t_1,\ldots,t_d)| \gtrsim \prod_{j=1}^d|L_\gamma(t_j)|^{\frac1d} \prod_{1 \leq i < j \leq d}|t_j-t_i|.
\end{equation}
The implicit constants in \eqref{E:convl gi} only depend on $C$, $M$, and $d$.    
\end{proposition}

The hypotheses on the functions $L_\gamma (L_\gamma^{d-k-1})^k (L_\gamma^{d-k})^{-k-1}$, $1\le k \le d-1$, are a direct generalization of the conditions in \cite{BOS2, OberlinJFA} for simple curves of the form $(t,t^2,\ldots,t^{d-1},\phi(t))$. For these curves, all the $L_\gamma^k$, $1\le k\le d-1$, are identically constant and $L_\gamma$ is a constant multiple of $\phi^{(d)}$. In that case, our hypotheses amount to a condition on the $d$-th derivative of $\phi$, not any lower derivatives, so they are slightly weaker than the hypotheses in \cite{BOS2, OberlinJFA}. For non-simple curves, easy examples in low dimensions (such as $\gamma(t) = (\cos t, \sin t)$ and $\gamma(t) = (\cos t, \sin t , t)$) show that conditions only on $L_\gamma$ are not sufficient. One can check that for $I = \R$ the geometric inequality \eqref{E:convl gi} fails, and the corresponding convolution and restricted X-ray transforms are unbounded.

We devote the bulk of this section to the proof of Proposition~\ref{P:RWT}, and we will indicate what changes need to be made in order to prove Proposition~\ref{P:RWTXray} at the end of the section.  

\begin{proof}[Proof of Proposition~\ref{P:RWT}]  
By standard approximation arguments, we may assume that $I$ is a compact interval and that $\gamma\in C^d(I;\R^d)$.  Since $f$ is log-concave, we may split $I$ into two open intervals $I_L$ and $I_R$ such that $f$ is increasing on $I_L$ and decreasing on $I_R$. We can reparametrize $t\mapsto -t$ on $I_R$, and therefore without loss of generality we assume that $f$ is increasing on $I$. Since the restricted weak type inequality is trivial on any interval on which $\lambda_\gamma$ is identically zero, and since $f$ is increasing, we may assume that $\lambda_\gamma$ is nonvanishing on $I$.  Thus we may assume that $f \in C^1$.  Finally, by a reparametrization, we may assume that $I=[0,1]$.  These assumptions will remain in force for the remainder of the argument.  

The following will be helpful later on.

\begin{lemma} \label{L:foh} For $0 \leq \rho \leq \int_0^1 \lambda_\gamma(t)\, dt$, define $h(\rho)$ to be the unique element of $I$ satisfying 
$$
\int_0^{h(\rho)} \lambda_\gamma(t)\, dt = \rho.
$$
Then for all $0 \leq a,b \leq \int_0^1\lambda_\gamma(t)\, dt$ and $0 \leq \theta \leq 1$,
\begin{equation} \label{E:foh}
\lambda_\gamma\circ h((1-\theta)a+\theta b)\gtrsim (1-\theta)\lambda_\gamma\circ h(a) + \theta \lambda_\gamma \circ h(b).
\end{equation}
\end{lemma}

\begin{proof}[Proof of Lemma~\ref{L:foh}]
We compute
\begin{equation} \label{E:foh'}
(f \circ h)'(\rho) = f'\circ h(\rho)h'(\rho) = \frac{f'\circ h(\rho)}{\lambda_\gamma\circ h(\rho)} \sim \frac{f'\circ h(\rho)}{f\circ h(\rho)} = (\log f)'\circ h.
\end{equation}
Since $\lambda_\gamma>0$, $h$ is strictly increasing and $\log f$ is concave, $(\log f)'\circ h$ is decreasing.  Therefore by \eqref{E:foh'}, $f\circ h$ is concave, and since $\lambda_\gamma \sim f$, \eqref{E:foh} follows.  
\end{proof}

Define quantities 
$$
\alpha = \tfrac{\langle T_\gamma \chi_E, \chi_F \rangle}{|F|}, \qquad \beta = \tfrac{\langle T_\gamma \chi_E, \chi_F \rangle}{|E|}.
$$
By simple arithmetic, \eqref{E:RWT} is equivalent to 
\begin{equation} \label{E:RWT equiv}
|F| \gtrsim \beta^d \alpha^{\frac{d^2-d}2}.
\end{equation}

The key step in the proof is the next lemma.  

\begin{lemma} \label{L:param set}
If $d$ is even, there exist $x_0 \in F$ and a Borel set $\Omega \subset I^d$ with the following properties:
$$
\int_\Omega \prod_{i=1}^d \lambda_\gamma(t_i)\, dt \gtrsim \alpha^{\frac d2}\beta^{\frac d2}; 
$$
if $(t_1,\ldots,t_d) \in \Omega$, then $0 < t_1 < t_2 < \cdots < t_d < 1$, and furthermore,
\begin{gather} \label{E:in F}
x_0+\sum_{i=1}^d (-1)^i \gamma(t_i) \in F,\\ \label{E:gtrsim ab}
\int_0^{t_1} \lambda_\gamma(t)\, dt \gtrsim \max\{\alpha,\beta\},\qquad  \int_{t_{i-1}}^{t_i} \lambda_\gamma(t)\, dt \gtrsim \begin{cases}\alpha, &\qtq{if $i>1$ is odd,}\\\beta, &\qtq{if $i$ is even.}\end{cases}
\end{gather}
If $d$ is odd, there exist a point $y_0 \in E$ and a Borel set $\Omega \subset I^d$ satisfying:
$$
\int_\Omega \prod_{i=1}^d \lambda_\gamma(t_i)\, dt \gtrsim \beta^{\frac{d+1}2}\alpha^{\frac{d-1}2};
$$
and moreover, if $(t_1,\ldots,t_d) \in \Omega$, then $0 < t_1 < t_2 < \cdots < t_d < 1$, and 
\begin{gather*}
y_0-\sum_{i=1}^d (-1)^i \gamma(t_i) \in F\\
\int_0^{t_1}\lambda_\gamma(t)\, dt \gtrsim \max\{\alpha,\beta\}, \qquad \int_{t_{i-1}}^{t_i} \lambda_\gamma(t)\, dt \gtrsim \begin{cases}\beta, &\qtq{if $i>1$ is odd,}\\\alpha,&\qtq{if $i$ is even.}\end{cases}
\end{gather*}
\end{lemma}

Before proving the lemma, we show how to complete the proof of the proposition.  This portion of the argument is much simpler than previous arguments such as \cite{ChCCC} because the ordering of the $t_i$ means that we can avoid the band structure argument.  That the $t_i$ might be ordered was suggested to the second author by Phil Gressman.

We assume for now that the dimension $d$ is even; the completion of the proof in the odd dimensional case is similar and will be left to the reader.  

Recall that we may assume that $L_\gamma$ is never zero on $I$.  This plus the ordering implies that $\Phi_\gamma$ is one-to-one on $\Omega$ (see \cite[Section~3]{DrM2}).  From \eqref{E:in F}, this injectivity, and the geometric inequality, we have
\begin{equation} \label{E:lb F}
|F| \gtrsim \int_\Omega \prod_{i=1}^d |L_\gamma(t_i)|^{\frac 1d}\prod_{i<j}|t_i-t_j|\, dt.
\end{equation}

Let $(t_1,\ldots,t_d) \in \Omega$.  Let $i>1$ be odd.  By \eqref{E:gtrsim ab}, there exists an $s_i$ with $t_{i-1}<s_i<t_i$ such that
$$
\int_{t_{i-1}}^{s_i} \lambda_\gamma(t)\, dt \sim \alpha.
$$
By \eqref{E:gtrsim ab} and the fact that $t_{i-1} \geq t_1$, this implies that
$$
\int_0^{s_i} \lambda_\gamma(t)\, dt = \theta \int_0^{t_{i-1}}\lambda_\gamma(t)\, dt,
$$
for some $\theta \sim 1$.
Thus by Lemma~\ref{L:foh}, 
\begin{align*}\lambda_\gamma(s_i) &\sim f(s_i) \ge f(t_{i-1}) \sim \lambda_\gamma(t_{i-1}) = \lambda_\gamma \circ h\Big(\int_0^{t_{i-1}}\lambda_\gamma(t) dt\Big) \\ 
&\gtrsim \lambda_\gamma \circ h\Big(\int_0^{s_i}\lambda_\gamma(t) dt\Big) =\lambda_\gamma(s_i). 
\end{align*}
Therefore
\begin{align*}
|t_i-t_{i-1}|&\geq |s_i-t_{i-1}| \sim \lambda_\gamma(s_i)^{-\frac12}\lambda_\gamma(t_{i-1})^{-\frac12} \int_{t_{i-1}}^{s_i}\lambda_\gamma(t)\, dt \\&\gtrsim \alpha \lambda_\gamma(t_{i-1})^{-\frac12}\lambda_\gamma(s_i)^{-\frac12} \gtrsim \alpha \lambda_\gamma(t_{i-1})^{-\frac12}\lambda_\gamma(t_i)^{-\frac12}.  
\end{align*}
Similar arguments show that if $i$ is even, 
$$
|t_i-t_{i-1}| \gtrsim \beta\lambda_\gamma(t_{i-1})^{-\frac12}\lambda_\gamma(t_i)^{-\frac12}
$$
and if $1 \leq i \leq j-2 \leq d-2$, then 
$$
|t_j-t_i| \gtrsim \alpha \lambda_\gamma(t_j)^{-\frac12}\lambda_\gamma(t_i)^{-\frac12}.
$$
Thus if $(t_1,\ldots,t_d) \in \Omega$,
$$
\prod_{1 \leq i<j\leq d} |t_j-t_i| \gtrsim \alpha^{\frac{d(d-1)}2-\frac d2}\beta^{\frac d2}\prod_{i=1}^d \lambda_\gamma(t_i)^{-\frac{d-1}2}, 
$$
so by \eqref{E:lb F} and some arithmetic (recall that $\lambda_\gamma = |L_\gamma|^{\frac2{d(d+1)}}$),
$$
|F| \gtrsim \alpha^{\frac{d(d-1)}2-\frac d2}\beta^{\frac d2}\int_\Omega \prod_{i=1}^d \lambda_\gamma(t_i)\, dt \gtrsim \alpha^{\frac{d(d-1)}2}\beta^d,
$$
which is just \eqref{E:RWT equiv}.  Therefore, all that is needed to establish Proposition~\ref{P:RWT} is to prove Lemma~\ref{L:param set}.  

\begin{proof}[Proof of Lemma~\ref{L:param set}]
Let 
$$
U = \{(x,t) \in F \times I : x-\gamma(t) \in E\}.
$$
Then
\begin{equation} \label{E:U = ip}
\langle T_\gamma \chi_E, \chi_F \rangle = \int \chi_U(x,t) \, \lambda_\gamma(t)\, dt\, dx.
\end{equation}
Define projections $\pi_1:U \to E$ and $\pi_2:U \to F$ by
$$
\pi_1(x,t) = x-\gamma(t), \qquad \pi_2(x,t) = x.
$$
This means that $U= \pi_1^{-1}(E) \cap \pi_2^{-1}(F)$.

We will make several refinements to $U$.  Define
$$
B_0 = \{(x,t) \in U : 0 < t < h(\tfrac14\alpha) \}, \qquad B_0' = \{(x,t) \in U : 0 < t < h(\tfrac14\beta)\}.
$$
Then by Fubini and the change of variables formula,
\begin{align*}
\int {\chi}_{B_0}(x,t) \lambda_\gamma(t)\, dt\, dx &= \int \chi_F(x) \int_0^{h(\tfrac14\alpha)} \chi_E(x-\gamma(t)) \lambda_\gamma(t)\, dt\, dx\\
&\le \int\chi_F(x) \tfrac14\alpha \, dx = \tfrac14\alpha|F| = \tfrac14\langle T_\gamma\chi_E,\chi_F\rangle,\\
\int\chi_{B_0'}(x,t)\lambda_\gamma(t)\, dt\, dx &= \int\chi_E(y)\int_0^{h(\tfrac14\beta)} \chi_F(y+\gamma(t)) \lambda_\gamma(t)\, dt\, dy\\
&<\tfrac14\beta|E| = \tfrac14\langle T_\gamma\chi_E,\chi_F \rangle.
\end{align*}
Therefore by \eqref{E:U = ip},
$$
\int\chi_{U \setminus(B_0 \cup B_0')}\lambda_\gamma(t)\, dt\, dx \geq \tfrac12 \int\chi_U(x,t) \lambda_\gamma(t)\, dt\, dx.
$$
Set $U_0=U\setminus(B_0\cup B_0')$.  Then
\begin{align} \label{E:size U0}
\tfrac12\langle T_\gamma\chi_E,\chi_F\rangle &\leq \int\chi_{U_0}(x,t)\, \lambda_\gamma(t)\, dt\, dx \leq \langle T_\gamma\chi_E,\chi_F\rangle,\\ \label{E:lambda U0}
\int_0^t \lambda_\gamma(s)\, ds &\geq \tfrac14\max\{\alpha,\beta\}, \qtq{for all} (x,t) \in U_0.
\end{align}

Define
$$
B_1 = \{(x,t) \in U_0 : \int_0^1 \chi_{U_0}(x-\gamma(t)+\gamma(s),s)\, \lambda_\gamma(s)\, ds < \tfrac1{8}\beta\}.
$$
Then by the change of variables formula and \eqref{E:lambda U0},
\begin{align*}
& \int \chi_{B_1}(x,t)\, \lambda_\gamma(t)\, dt\, dx = \int \chi_{B_1}(y+\gamma(t),t)\, \lambda_\gamma(t)\, dt\, dy\\
&\qquad  \leq \int_{\{y \in E : T^*\chi_F(y)<\frac1{8}\beta\}} T^*\chi_F(y)\, dy
  < \tfrac1{8}\beta |E|<\tfrac12 \int\chi_{U_0}(x,t)\, \lambda_\gamma(t)\, dt\, dx.
\end{align*}
Up to now, this has been the usual procedure in the method of refinements. We must take further care to ensure that the $t_i$ are ordered by removing points $(x,t)$ for which $t$ is too large. To this end, let $U_1' = U_0 \setminus B_1$ and define
$$
B_1' = \{(x,t) \in U_1' : \int_t^1 \chi_{U_1'}(x-\gamma(t)+\gamma(s),s)\, \lambda_\gamma(s)\, ds < \tfrac1{16}\beta\}.
$$
If $y \in \pi_1(U_1')$,
$$
\int \chi_{U_1'}(y+\gamma(s),s)\, \lambda_\gamma(s)\, ds = \int \chi_{U_0}(y+\gamma(s),s)\, \lambda_\gamma(s)\, ds \geq \tfrac18\beta,
$$
and
$$
\int\chi_{B_1'}(y+\gamma(s),s)\lambda_\gamma(s)\, ds = \tfrac1{16}\beta \leq \tfrac12 \int\chi_{U_1'}(y+\gamma(s),s)\, \lambda_\gamma(s)\, ds.
$$
Therefore if $U_1:= U_1'\setminus B_1'$,
\begin{align*}
\int\chi_{U_1}(x,t)\, \lambda_\gamma(t)\, dt \, dx &= \int_{\pi_1(U_1)}\int\chi_{U_1}(y+\gamma(s),s)\, \lambda_\gamma(s)\, ds\, dy\\
& \geq \tfrac12 \int_{\pi_1(U_1')} \int\chi_{U_1'}(y+\gamma(s),s) \, \lambda_\gamma(s)\, ds\, dy
\\&= \tfrac12\int\chi_{U_1'}(x,t)\lambda_\gamma(t)\, dt \, dx \geq \tfrac1{16} \int\chi_U(x,t)\lambda_\gamma(t)\, dt\, dx.
\end{align*}
In short, if $(x,t)\in U_1$, then there is a set $A(x,t) \subset (t,1)$ such that
\[
\int_{A(x,t)} \lambda_\gamma (s) ds \geq \tfrac{1}{16}\beta, \quad \tnorm{and} \quad s\in A(x,t) \Rightarrow (x-\gamma(t) +\gamma(s),s)\in U_0.
\]
It is the fact that $A(x,t) \subset (t,1)$ that will allow us to ensure that the $t_i$ are increasing when we form the set $\Omega_d$.

To continue this process, we define
\begin{align*}
B_2 &= \{(x,t) \in U_1 : \int_0^1 \chi_{U_1}(x,s)\, \lambda_\gamma(s)\, ds < \tfrac1{32}\alpha\}, \qquad  U_2' = U_1\setminus B_2,\\
B_2' &= \{(x,t) \in U_2' : \int_t^1 \chi_{U_2'}(x,s)\, \lambda_\gamma(s)\, ds < \tfrac1{64}\alpha\}, \qquad U_2= U_2' \setminus B_2'.
\end{align*}
Arguing as above, 
$$
\int \chi_{U_2}(x,t)\, \lambda_\gamma(t)\, dt\, dx \geq \tfrac1{64}\int\chi_U(x,t)\, \lambda_\gamma(t)\, dt\, dx.
$$

We iterate.  Assume that the set $U_k \subset U$ satisfying 
$$
\int \chi_{U_k}(x,t)\, \lambda_\gamma(t)\, dt\, dx \geq \tfrac1{4^{k+1}}\int\chi_U(x,t)\, \lambda_\gamma(t)\, dt\, dx
$$
has been constructed.  If $k$ is even, we define
\begin{align*}
B_{k+1} &= \{(x,t)\in U_k : \int_0^1 \chi_{U_k}(x-\gamma(t)+\gamma(s),s)\, \lambda_\gamma(s)\, ds < \tfrac1{4^{k+3/2}}\beta\}, \\
U_{k+1}' &= U_k\setminus B_{k+1},\\
B_{k+1}' &= \{(x,t) \in U_{k+1}': \int_t^1\chi_{U_{k+1}'}(x-\gamma(t)+\gamma(s),s)\, \lambda_\gamma(s)\, ds < \tfrac1{4^{k+2}}\beta\},\\
U_{k+1} &= U_{k+1}'\setminus B_{k+1}',
\end{align*}
so if $(x,t) \in U_{k+1}$,
$$
\int_t^1 \chi_{U_k}(x-\gamma(t)+\gamma(s),s)\, \lambda_\gamma(s)\, ds \geq \tfrac1{4^{k+2}}\beta, \quad \text{$k$ even.}
$$
If $k$ is odd, we define
\begin{align*}
B_{k+1}  &= \{(x,t) \in U_k : \int_0^1 \chi_{U_k}(x,s)\, \lambda_\gamma(s)\, ds < \tfrac1{4^{k+3/2}}\alpha\},\\
U_{k+1}' &= U_k \setminus B_{k+1},\\
B_{k+1}' &= \{(x,t) \in U_{k+1}' : \int_t^1 \chi_{U_{k+1}'}(x,s)\, \lambda_\gamma(s)\, ds < \tfrac1{4^{k+2}}\alpha\},\\
 U_{k+1} &= U_{k+1}'\setminus B_{k+1}',
\end{align*}
so if $(x,t) \in U_{k+1}$,
$$
\int_t^1 \chi_{U_k}(x,s)\, \lambda_\gamma(s)\, ds \geq \tfrac1{4^{k+2}}\alpha, \quad \text{$k$ odd.}
$$
Similar arguments to those for $k=1,2$ show that 
$$
\int \chi_{U_k}(x,t)\, \lambda_\gamma(t)\, dt\, dx \geq \tfrac1{4^{k+1}}\int\chi_U(x,t)\, \lambda_\gamma(t)\, dt\, dx, \quad \text{for all $k$.}
$$
In particular, $U_d$ is nonempty.

If the sets $E$ and $F$ are Borel, then $U$ is Borel, as are each of the refinements $U_k$, so measurability is not an issue.  

At this point, the arguments when $d$ is even and $d$ is odd diverge.  We give the details when $d$ is even; the proof when $d$ is odd is essentially the same.  Let $(x_0,t_0) \in U_d$.  We will construct a sequence of sets, $\Omega_1 \subset I$, $\Omega_k \subset \Omega_{k-1}\times I$, $2 \leq i \leq d$; $\Omega = \Omega_d$ will be the set whose existence was claimed in Lemma~\ref{L:param set}.

We construct the $\Omega_k$ inductively.  Let $(t_1,\ldots,t_k) \in \Omega_k$, and define
$$
(x_k,t_k) = \begin{cases} (x_0+\sum_{j=1}^{k-1}(-1)^j \gamma(t_j),t_k), \quad&\text{$k$ odd;}\\ (x_0+\sum_{j=1}^k (-1)^j \gamma(t_j),t_k), \quad &\text{$k$ even.}\end{cases}
$$
The $\Omega_k$ will be defined so that $(x_k,t_k) \in U_{d-k}$, $0 \leq k \leq d$.  In particular, $x_0 \in F$ and if $(t_1,\ldots,t_d) \in \Omega_d$, then $x_d = x_0+\sum_{j=1}^d (-1)^j \gamma(t_j) \in F$.  

Since $(x_k,t_k) \in U_{d-k}$ (and since $d$ is even),
\begin{equation} \label{E:after tk}
\begin{aligned}
\int_{t_k}^1 \chi_{U_{d-k-1}}(x_k,s) \, \lambda_\gamma(s)\, ds \geq \tfrac1{4^{d-k+1}}\alpha, \quad &\text{ $k$  even,}\\
\int_{t_k}^1 \chi_{U_{d-k-1}}(x_k-\gamma(t_k)+\gamma(s),s)\, \lambda_\gamma(s)\, ds \geq \tfrac1{4^{d-k+1}}\beta, \quad &\text{ $k$  odd}.
\end{aligned}
\end{equation}
Choose $s_{k+1} \geq t_k$ such that
$$
\int_{t_k}^{s_{k+1}}\lambda_\gamma(s)\, ds =
\begin{cases} \tfrac1{4^{d-k+3/2}}\alpha, \quad &\text{ $k$  even,}\\ \tfrac1{4^{d-k+3/2}}\beta, \quad &\text{ $k$  odd.}\end{cases}
$$
Define
\begin{align*}
\Omega_{k+1}(x_k,t_k) = \begin{cases} \{t_{k+1} \geq s_{k+1} : (x_k,t_{k+1}) \in U_{d-k-1}\},  &\text{ $k$  even,}\\
 \{t_{k+1} \geq s_{k+1} : (x_k-\gamma(t_k)+\gamma(t_{k+1}),t_{k+1}) \in U_{d-k-1}\},  &\text{ $k$  odd}.
 \end{cases}
\end{align*}
By \eqref{E:after tk},
$$
\int_{\Omega_{k+1}(x_k,t_k)}\lambda_\gamma(t_{k+1})\, dt_{k+1} \geq \begin{cases} \tfrac1{4^{d-k+3/2}}\alpha, \quad &\text{ $k$  even,}\\
 \tfrac1{4^{d-k+3/2}}\beta, \quad &\text{ $k$  odd}.\end{cases}
 $$
Finally, define $\Omega_1 = \Omega_1(x_0,t_0)$, and if $k \geq 1$, define
$$
\Omega_{k+1} = \{(t_1,\ldots,t_k,t_{k+1}) : (t_1,\ldots,t_k) \in \Omega_k, \: t_{k+1} \in \Omega_{k+1}(x_k,t_k)\}.
$$

The final set, $\Omega=\Omega_d$ now has all of the properties claimed in the lemma.  This completes the proof of Lemma~\ref{L:param set}.
\end{proof}
The proof of Proposition~\ref{P:RWT} is also now complete.  
\end{proof}

We now turn to the restricted X-ray transform.  In establishing bounds for $X$ via the method of refinements, the maps that arise are slightly more complicated.  Given base points $(s_0,x_0), (t_0,y_0) \in \R^{1+d}$, we define maps $\Phi^k_{(s_0,x_0)}, \linebreak \Psi^k_{(t_0,y_0)}:\R^k \to \R^{1+d}$ ($k=1,2,\dots,d+1$) by
\begin{align}
\label{E:def Phi even}
&\Phi^{2K}_{(s_0,x_0)}(t_1,s_1,\dots,t_K,s_K)   = \bigl(s_K,x_0 - \sum_{j=1}^K(s_{j-1}-s_j)\gamma(t_j)\bigr), \\
\label{E:def Phi odd} 
&\Phi^{2K+1}_{(s_0,x_0)}(t_1,s_1,\dots,t_{K+1})  = \bigl(t_{K+1},x_0 - \sum_{j=1}^K(s_{j-1}-s_j)\gamma(t_j) -s_K\gamma(t_{K+1})\bigr), \\
\label{E:def Psi even}
&\Psi^{2K}_{(t_0,y_0)}(s_1,t_1,\dots,s_K,t_K)  = \bigl(t_K,y_0 + \sum_{j=1}^Ks_j(\gamma(t_{j-1}) - \gamma(t_j))\bigr), \\
\label{E:def Psi odd}
&\Psi^{2K+1}_{(t_0,y_0)}(s_1,t_1,\dots,s_{K+1}) = \bigl(s_{K+1},y_0 + s_1\gamma(t_0)- \sum_{j=1}^K (s_j-s_{j+1})\gamma(t_j)\bigr).
\end{align}

We have the following geometric inequalities.   

\begin{proposition} \label{P:jacobian bounds} Let $\gamma:I \to \R^d$ be a $C^d$ curve satisfying the geometric inequality \eqref{E:convl GI} on $I^d$.  

\emph{(i)} If $d+1 = 2D$ is an even integer, $(s_0,s_1,\dots,s_D) \in \R^{D+1}$ and $(t_0,t_1,\dots,t_D) \in I_j^{D+1}$, then
\begin{align} \label{E:Psi even}
&|\det\bigl(D\Psi^{d+1}_{(t_0,y_0)}(s_1,t_1,\dots,s_D,t_D)\bigr)| \\ \notag
&\qquad  \gtrsim \prod_{i=1}^{D-1}\bigl\{|s_{i+1}-s_i||L_\gamma(t_i)|^{\frac2{d+1}}\prod_{\stackrel{0 \leq j \leq D}{j\neq i}}|t_j-t_i|^2\bigr\} |L_\gamma(t_0)|^{\frac1{d+1}}|L_\gamma(t_D)|^{\frac1{d+1}}|t_D-t_0| \\
 \label{E:Phi even}
& \begin{aligned}
 &|\det\bigl(D\Phi^{d+1}_{(s_0,x_0)}(t_1,s_1,\dots,t_D,s_D)\bigr)| \\
 &\qquad \gtrsim \prod_{i=1}^D \bigl\{|s_i-s_{i-1}||L_\gamma(t_i)|^{\frac2{d+1}}\prod_{\stackrel{1 \leq j \leq D}{j\neq i}}|t_j-t_i|^2\bigr\} 
 \end{aligned}
\end{align}

\emph{(ii)} If $d+1=2D+1$ is odd, then analogous statements hold, only we must modify the bounds in \eqref{E:Psi even}, \eqref{E:Phi even} to
\begin{align} \label{E:Psi odd}
&\begin{aligned}
&|\det\bigl(D\Psi^{d+1}_{(t_0,y_0)}(s_1,t_1,\dots,s_{D+1})\bigr)| \\ 
&\qquad  \gtrsim \prod_{i=1}^D\bigl\{|s_{i+1}-s_i||L_\gamma(t_i)|^{\frac2{d+1}}\prod_{\stackrel{0 \leq j \leq D}{j \neq i}}|t_j-t_i|^2\bigr\} |L_\gamma(t_0)|^{\frac1{d+1}}
\end{aligned}\\
\label{E:Phi odd}
&\begin{aligned}
& |\det\bigl(D\Phi^{d+1}_{(s_0,x_0)}(t_1,s_1,\dots,t_{D+1})\bigr)|  \\ 
& \qquad \gtrsim \prod_{i=1}^D \bigl\{|s_i-s_{i-1}| |L_\gamma(t_i)|^{\frac2{d+1}}\prod_{\stackrel{1 \leq j \leq D+1}{j \neq i}}|t_j-t_i|^2 \bigr\} |L_\gamma(t_{D+1})|^{\frac1{d+1}}.
\end{aligned}
\end{align}
Here again, $t_i \in I_j$, while $s_i \in \R$.
\end{proposition}

(Technically, only two of these arise, but we give all inequalities for the possible convenience of the reader.)  

We also need an almost-injectivity result for the maps in (\ref{E:def Phi even}-\ref{E:def Phi odd}). It will be easier to state if we abuse notation and write $\Phi_{(s_0,x_0)}^{d+1}(t_1,s_1,\ldots) = \Phi_{(s_0,x_0)}^{d+1}(t,s)$, $\Psi_{(t_0,y_0)}^{d+1}(s_1,t_1,\ldots) = \Phi_{(t_0,y_0)}^{d+1}(t,s)$.

\begin{proposition} \label{P:one to one} Let $\gamma:I \to \R^d$ be a $C^1$ curve and assume that $J_\gamma$ is nonzero on $\{t \in I^d : t_1 < \cdots < t_d\}$.  Then each map $\Phi_{(s_0,x_0)}^{d+1}$, $\Psi_{(t_0,y_0)}^{d+1}$ is at most $(D+1)!$-to-one on
$$
\Delta := \{(t,s) : \text{$t_0,t_1,\ldots \in I$ are distinct, and $s_0,s_1,\ldots \in \R$ are distinct}\}
$$
\end{proposition}

We note in particular that the hypotheses are satisfied whenever $\gamma$ is a $C^d$ curve satisfying the geometric inequality \eqref{E:convl GI} on $I^d$ and $L_\gamma \neq 0$ on $I$.  

We will prove Propositions~\ref{P:jacobian bounds} and~\ref{P:one to one} in Section~\ref{S:gi proofs}.  Assuming their validity for now, we outline the changes that must be made to the proof of Proposition~\ref{P:RWT} in order to establish Proposition~\ref{P:RWTXray}.  By a simple arithmetic argument, the key step in the proof by refinements is the following.

\begin{lemma} \label{L:param setXray}
If $d+1=2D \ge 4$ is even, then there exist a point $(t_0,y_0) \in F$ and a Borel set $\Omega \subset \R^{d+1}$ such that 
$$
\int_\Omega \prod_{i=0}^{D} \lambda_\gamma(t_i)\, dt_D\, ds_D \ldots dt_1\, ds_1 \gtrsim \alpha^D \beta^D
$$
and such that if $(s_1,t_1,\ldots,s_D,t_D) \in \Omega$, then $0 < t_0 < t_1 < \cdots < t_D < 1$, $s_1 < \cdots < s_D$ and, for $2\le i\le D$,
\begin{gather*}
\Psi^{d+1}_{(t_0,y_0)}(\Omega) \in F, \quad \lambda_\gamma(t_{i-1}) (s_i -s_{i-1}) \gtrsim \alpha,\\ 
\int_0^{t_0} \lambda_\gamma(t)\, dt \gtrsim \beta, \quad \int_{t_0}^{t_1} \lambda_\gamma(t)\, dt \gtrsim \beta, \quad \int_{t_{i-1}}^{t_i} \lambda_\gamma(t)\, dt \gtrsim \beta.
\end{gather*}
If $d+1=2D+1 \ge 3$ is odd, then there exist a point $(s_0,x_0) \in F$ and a Borel set $\Omega \subset \R^{d+1}$ such that 
$$
\int_\Omega \prod_{i=1}^{D+1} \lambda_\gamma(t_i)\,dt_{D+1}\, ds_D\,dt_D\ldots ds_1\, dt_1 \gtrsim \alpha^D \beta^{D+1}
$$
and such that if $(t_1,s_1,\ldots,t_D,s_D,t_{D+1}) \in \Omega$, then $0 < t_1 < t_2 < \cdots < t_{D+1} < 1$, $s_0 < \cdots < s_D$, and for $2\le i\le D$,
\begin{gather*}
\Phi^{d+1}_{(s_0,x_0)}(\Omega) \in F, \quad \lambda_\gamma(t_1) (s_1-s_0) \gtrsim \alpha, \quad \lambda_\gamma(t_i) (s_i -s_{i-1}) \gtrsim \alpha,\\ 
\int_0^{t_1} \lambda_\gamma(t)\, dt \gtrsim \beta, \quad \int_{t_{i-1}}^{t_i} \lambda_\gamma(t)\, dt \gtrsim \beta.
\end{gather*}
\end{lemma}

\begin{proof}[Sketch of proof]
Lemma \ref{L:param setXray} may be proved similarly to Lemma~\ref{L:param set}, but with some adjustments.  Define
$$
U := \{(s,t,y) \in \R \times I \times \R^d : (s,y+s\gamma(t)) \in E, \: (t,y) \in F\}.
$$
Then 
$$
\langle X_\gamma\chi_E,\chi_F \rangle = \int_U \lambda_\gamma(t)\, ds\, dt\, dy.
$$
We define 
\begin{align*}
B_0 &:= \{(s,t,y) \in U : t < h(\tfrac12\beta)\}\\
U_0 &:= U\setminus B_0.
\end{align*}
It is easy to check that $\int_{U_0}\lambda_\gamma(t)\, ds\, dt\, dy \geq \tfrac12 \int_U \lambda_\gamma(t)\, ds\, dt\, dy$.  

We will again iteratively construct a sequence of sets $U_0 \supset U_1 \supset U_2 \supset \cdots$.  Assume that $U_k$ satisfying $\int_{U_k} \lambda_\gamma(t)\, ds\, dt\, dy \geq \tfrac1{4^{k+1}}\int_U \lambda_\gamma(t)\, ds\, dt\, dy$ is given.  We define bad sets to be excised.  If $k$ is even, these are
\begin{align*}
B_{k+1} &= \{(s,t,y) \in U_k : \int_0^1 \chi_{U_k}(s,\tau,y+s\gamma(t)-s\gamma(\tau))\, \lambda_\gamma(\tau)\, d\tau < \tfrac1{2^{2k+3}}\beta\}\\
B_{k+1}' &= \{(s,t,y)\in U_k\setminus B_{k+1} : \int_t^1 \chi_{U_k}(s,\tau,y+s\gamma(t)-s\gamma(\tau))\, \lambda_\gamma(\tau)\, d\tau < \tfrac1{2^{2k+4}}\beta\}.
\end{align*}
If $k$ is odd, we define
\begin{align*}
B_{k+1} &= \{(s,t,y) \in U_k : \int_0^1 \chi_{U_k}(\sigma,t,y)\, \lambda_\gamma(t)\, d\sigma < \tfrac1{2^{2k+3}}\alpha\}\\
B_{k+1}' &= \{(s,t,y) \in U_k\setminus B_{k+1} : \int_s^1 \chi_{U_k}(\sigma,t,y)\, \lambda_\gamma(t)\, d\sigma < \tfrac1{2^{2k+3}}\alpha\}.
\end{align*}
In either case, the next set is $U_{k+1} := U_k \setminus (B_{k+1}\cup B_{k+1}')$.  It may be verified as in the proof of Lemma~\ref{L:param set} that $U_{k+1}$ satisfies the inductive hypothesis $\int_{U_{k+1}}\lambda_\gamma(t)\, ds\, dt\, dy \geq \tfrac1{4^{k+2}} \int_U\lambda_\gamma(t)\, ds\, dt\, dy$.  

We leave the remaining details to the reader.
\end{proof}

The rest of the proof of Proposition \ref{P:RWTXray} uses Lemma \ref{L:foh} and the same argument as the proof of Proposition \ref{P:RWT}.

\section{A simple interpolation argument} \label{S:interpolation}

In this section, we show how restricted weak type endpoint bounds for differently weighted operators may be deduced from the restricted weak type endpoint bounds operators with affine arclength measure.  We work with restricted weak type estimates;  somewhat related arguments may be found in \cite{GSW, RicciStein}.  The general form of our result follows.  

\begin{proposition} \label{P:interpolation}
Let $I$ be an interval and let $\gamma \in C^d_{\rm{loc}}(I;\R^d)$, and assume that $\{t \in I : L_\gamma(t) = 0\}$ has measure 0.  If $\langle T_\gamma \chi_E,\chi_F\rangle \leq C |E|^{\frac1{p_d}}|F|^{1-\frac1{q_d}}$ for Borel sets $E,F \subset \R^d$, then
$$
\langle T_\gamma^\theta \chi_E,\chi_F \rangle \lesssim_{C,\theta} |E|^{\frac\theta{p_d}} |F|^{1-\frac\theta{q_d}}, \qquad 0 < \theta \leq 1,
$$
where 
\begin{equation} \label{E:T theta}
T_\gamma^\theta f(x) := \int_I f(x-\gamma(t)) |\phi_\gamma(t)|^{\theta-1}\lambda_\gamma(t)\, dt, 
\end{equation}
and $\phi_\gamma(t) := \int_{t_0}^t\lambda_\gamma(s)\, ds$ for some arbitrary $t_0 \in I$.  
\end{proposition}

The proof is elementary.

\begin{proof}  We may assume that $t_0=0\in I \subseteq [0,\infty)$.  Since \eqref{E:T theta} is completely parametrization invariant, we reparametrize by $\phi_\gamma^{-1}$, so $\lambda_\gamma \equiv 1$.  Thus
$$
T_\gamma^\theta f(x) = \int_I f(x-\gamma(t)) \, \tfrac{dt}{t^{1-\theta}}.
$$
Let $I_n = I \cap [2^n,2^{n+1}]$ and $T_n f(x) = \int_{I_n} f(x-\gamma(t))\, dt$.  Then
\begin{align*}
\langle T_\gamma^\theta \chi_E,\chi_F\rangle &\leq 2 \sum_n 2^{-n(1-\theta)} \langle T_n \chi_E,\chi_F\rangle\\
&\leq 2\sum_n \min\{C 2^{-n(1-\theta)} |E|^{\frac1{p_d}}|F|^{1-\frac1{q_d}}, 2^n|F|\}\\
&\lesssim_{C,\theta} |E|^{\frac\theta{p_d}}|F|^{1-\frac\theta{q_d}}.
\end{align*}
\end{proof}

A similar result can easily be deduced for the restricted X-ray transform.  This raises the question of what the natural interpolants would be in the general translation non-invariant case considered in \cite{BSmult, TW}, and whether it is possible to obtain the unweighted estimates from the weighted ones using a similarly simple interpolation argument.

\subsection*{Examples} 
Let $\gamma(t) = (t^{a_1}\theta_1(t),\ldots,t^{a_d}\theta_d(t))$ be a monomial-like curve as described in \ref{SS:ST} of the introduction, and let $A=\sum_i a_i$.  It is proved in \cite{DMtams} (cf.\ Proposition~\ref{P:convl exp gi} in this article) that there exists $\tau \in \R$ such that the exponential parametrization
$$
\Gamma(t) = \gamma(e^{-t}) = (e^{-a_1 t}\Theta_1(t),\ldots,e^{-a_dt}\Theta_d(t))
$$
satisfies the following for all $t \in (\tau,\infty)$ and $(t_1,\ldots,t_d) \in (\tau,\infty)^d$:
\begin{gather*}
|L_\Gamma(t)| \sim c(a,\theta) e^{-At}, \qquad c(a,\theta) := \prod_{i=1}^d |a_i \theta_i(0)| \prod_{i<j}(a_j-a_i)\\
|J_\Gamma(t_1,\ldots,t_d)| \gtrsim \prod_{j=1}^d |L_\gamma(t_j)|^{\frac1d} \prod_{i<j} |t_j-t_i|,
\end{gather*}
with implicit constants depending only on $d$.  By Proposition~\ref{P:RWT}, we obtain endpoint restricted weak type estimates with implicit constants depending only on $d$.  By the above interpolation argument we therefore obtain restricted weak type endpoint bounds for
\begin{align*}
T_\gamma^\theta f(x) = \int_0^{e^{-\tau}} f(x-\gamma(t))\, \lambda_\gamma(t)^\theta \bigl(\tfrac{A}t\bigr)^{1-\theta}\, dt\qquad \text{if $A\neq 0$}\\
T_\gamma^\theta f(x) = \int_0^{e^{-\tau}} f(x-\gamma(t))\, \lambda_\gamma(t)^\theta (t|\log t|)^{\theta-1}\, dt\qquad \text{if $A=0$},
\end{align*}
with constants depending only on $d$.  

We will give strong type endpoint bounds for these operators in the next section, but those obtained by interpolation are more uniform in the case $A \neq 0$ (they are actually weaker in the case $A=0$).  In the special case $A\geq \tfrac{d^2+d}2$, $\theta = \tfrac{d^2+d}{2A}$ we recover the restricted weak type endpoint bounds for the unweighted operator.  (In the case of monomial curves with positive integer powers, the estimates below the unweighted endpoint were already seen in \cite{GressMonomial}, with an additional dependence on the degree.)

Let $\gamma(t) = (e^{-1/t},te^{-1/t})$.  For sufficiently small $\delta$, $\gamma$ satisfies the hypotheses of Proposition~\ref{P:RWT} on $(0,\delta)$, so the restricted weak type estimate \eqref{E:RWT} holds.  We compute
$$
T^\theta_\gamma  f(x) \sim \int_0^\delta f(x-\gamma(t)) t^{\frac23 \theta-2}e^{-\frac{2\theta}{3t}} dt,
$$
for $f \geq 0$, and by the proposition, we can bound this for $0 < \theta \leq 1$.  Reparametrizing, $\tilde \gamma(s) = (s,(\log s^{-1})^{-1} s)$ and 
$$
T^\theta_\gamma f(x) = \int_0^{e^{-1/\delta}} f(x-\tilde\gamma(s))\, (\log s^{-1})^{-2\theta/3}s^{2\theta/3-1}\, ds.
$$
This parametrization is somewhat more natural from a geometric viewpoint since $\tilde\gamma'(0) \neq 0$.  

Now we consider the opposite extreme.  The curve $\gamma(t) = (t,e^{-\frac1t})$ is infinitely flat at 0 (and may be viewed as a limiting case for the sequence of curves $(t,t^n)$).  For $f \geq 0$, we estimate
$$
T^\theta_\gamma f(x) \sim \int_0^\delta f(x-\gamma(t))\, t^{\frac{2\theta}3 -2}e^{-\frac\theta t}\, dt.
$$
By Propositions~\ref{P:RWT} and~\ref{P:convl gi}, $T^1_\gamma$ does satisfy the restricted weak type estimate \eqref{E:RWT} for sufficiently small $\delta$.  By the proposition, we can bound $T^\theta_\gamma$ for all $0 < \theta \leq 1$.   

In the limiting case $\theta \searrow 0$, $p_\theta = q_\theta = \infty$ and $q_\theta'=p_\theta'=1$.  For certain curves, the limiting operator is the one-sided (and hence ill-defined) Hilbert transform, but it seems interesting that other, even more singular integrals arise in this way.

\section{Strong type bounds}
\label{S:ST}

We now turn to the proof of Theorem \ref{T:convl st}.  Recall that $\gamma(t) = (t^{a_1}\theta_1(t),\ldots,\linebreak t^{a_d}\theta_d(t))$, $0 < t < T$, for nonzero (but possibly negative) real numbers $a_1 < \cdots < a_d$ and $\theta_i \in C^d_{\rm{loc}}((0,T))$, satisfying 
$$
\lim_{t \searrow 0} \theta_i(t) = \theta_i(0) \in \R\setminus\{0\}, \qquad \lim_{t \searrow 0} t^m \theta_i^{(m)} (t) = 0, \quad 1 \leq i,m \leq d.
$$
We use the exponential parametrization from \cite{DMtams}.  Set 
$$
\Gamma(t) = \gamma(e^{-t}), \qquad \Theta_i(t) = \theta_i(e^{-t}), \qquad \Theta_i(\infty) = \theta_i(0).
$$
Changing variables and setting $\tau = -\log\delta$,
\begin{equation} \label{E:exp T}
T_\gamma^\theta f(x) = \int_\tau^\infty f(x-\Gamma(t))\lambda_\Gamma(t)^\theta\, dt.
\end{equation}
Our main tool in bounding this operator will be the following geometric inequality.  

\begin{proposition} \label{P:convl exp gi}
There exists a constant $c_d >0$ such that for each $\gamma$ as above, there exists $\tau=\tau_\gamma \in \R$ such that 
\begin{equation} \label{E:LGamma}
| L_\Gamma(t)| \sim e^{-t\sum_{j=1}^d a_i} \prod_{j=1}^d|a_j \Theta_j(\infty)| \prod_{1 \leq i < j \leq d}(a_j-a_i)
\end{equation}
on $[\tau,\infty)$ and such that for all $(t_1,\ldots,t_d) \in (\tau,\infty)^d$, 
\begin{equation} \label{E:convl exp gi}
|J_\Gamma(t_1,\ldots,t_d)| \gtrsim \prod_{i=1}^d |L_\Gamma(t_i)|^{\frac1d}\prod_{1 \leq i < j \leq d} (|t_i-t_j|e^{c_d(a_d-a_1)|t_i-t_j|}).
\end{equation}
The implicit constants in \eqref{E:LGamma} and \eqref{E:convl exp gi} depend only on $d$.  
\end{proposition}

This represents a small improvement over the corresponding lower bound in \cite{DMtams} because of the presence of the exponential term on the far right of \eqref{E:convl exp gi}.  Though slight, this extra growth will be essential for our argument.  We will prove this proposition in Section~\ref{S:gi proofs}.  Now we concentrate on the proof of Theorem~\ref{T:convl st}, to which we devote the remainder of the section.  

By rescaling $t$ in \eqref{E:exp T} and using the parametrization invariance, we may assume that $\sum_{i=1}^d a_i = A \in  \{-\frac{d(d+1)}2,0,\frac{d(d+1)}2\}$ and that if $A=0$, $a_d-a_1=1$.  This assumption will remain in force for the remainder of the section.  The quantity $\tfrac A{a_d-a_1}$ is invariant under rescaling, so our goal is now to prove \eqref{E:convl st} with implicit constants depending on $d$, $\theta$, and an upper bound for $(a_d-a_1)^{-1}$.

For $\tau \leq j+1 \in \Z$ and $0 \leq \theta \leq 1$, we define
$$
T_j^\theta f(x) = \int_j^{j+1} f(x-\Gamma(t)) \lambda_\Gamma(t)^\theta \, dt,
$$
with the natural modification (which we will gloss over) when $j \leq \tau \leq j+1$.  

Using an extended method of refinements (in the spirit of \cite{ChQex}), we will show at the end of this section that Theorem~\ref{T:convl st} follows from the next two lemmas.  

\begin{lemma} \label{L:lbE0}
Let $E_1,E_2,F \subset \R^d$ be Borel sets and let $j_1,j_2 > \tau$ be integers.  For $i=1,2$, let $\beta_1 = \frac{\langle T^1_{j_1} \chi_{E_i},\chi_F\rangle}{|E_1|}$ and assume that $T_{j_i}^1 \chi_{E_i}(x) \geq \alpha_i$ for $x \in F$.  Then
\begin{equation}\label{E:lbE}
|E_2| \gtrsim e^{c_d(a_d-a_1)|j_1-j_2|} \alpha_1^{\frac{(d-1)^2}2 + \delta} \beta_1^{d-1}\alpha_2^{\frac{d+1}2-\delta},
\end{equation}
where $\delta = C_d^{-1}\min\{1,\tfrac{a_d-a_1}{|A|}\}$, and we interpret $\tfrac{a_d-a_1}0$ to be $+\infty$. The implicit constant may be taken to depend only on $d$.
\end{lemma}

\begin{lemma} \label{L:lbF0}
Let $E,F_1,F_2 \subset \R^d$ be Borel sets and let $j_1,j_2 > \tau$ be integers.  For $i=1,2$, let $\alpha_i = \frac{\langle T_{j_i}^1 \chi_E,\chi_{F_i}\rangle}{|F_i|}$ and assume that $(T_{j_i}^1)^* \chi_{F_i}(y) \geq \beta_i$ for $y \in E$.  Then if $|j_1-j_2| \leq 1$,
\begin{equation} \label{E:lbF same scale}
|F_2| \gtrsim \alpha_1^{\frac{d(d-1)}2}\beta_1^{d-\frac32+\eta}\beta_2^{\frac32-\eta},
\end{equation}
where $\eta = \tfrac12$ if $d=2,3$ and $\eta=0$ if $d \geq 4$, while if $|j_1-j_2| \geq 2$, 
\begin{equation} \label{E:lbF diff scales}
|F_2| \gtrsim e^{c_d(a_d-a_1)|j_1-j_2|} \alpha_1^{\frac{d(d-1)}2-\frac{d-2}2}\alpha_2^{\frac{d-2}2}\beta_1^{d-\frac32}\beta_2^{\frac32} .
\end{equation}
The implicit constants in \eqref{E:lbF same scale} and \eqref{E:lbF diff scales} may be taken to depend only on $d$.
$\,$
\end{lemma}

Lemmas~\ref{L:lbE0} and~\ref{L:lbF0} will be shown to follow from the geometric inequality, careful counting, and the following two lemmas.  Let 
\[w(a) = \prod_{j=1}^d|a_j \Theta_j(\infty)|^{\frac 2{d(d+1)}} \prod_{i<j} (a_j-a_i)^{\tfrac2{d(d+1)}}\] 
(which we recall is the constant in front of the affine arclength).  

\begin{lemma} \label{L:lbE}
Under the hypotheses of Lemma~\ref{L:lbE0}, there exists a Borel set $\Omega_d \subset [j_1,j_1+1]^{d-1} \times [j_2,j_2+1]$ with
\begin{equation} \label{E:lbE Omegad}
\idotsint \chi_{\Omega_d}(t_1,\cdots,t_d) \lambda_\Gamma(t_1)\cdots \lambda_\Gamma(t_d) \, dt_1\, \cdots \, dt_d \gtrsim \alpha_1^{\lceil \frac{d}2 \rceil-1}\alpha_2 \beta_1^{\lfloor\frac{d}2\rfloor}
\end{equation}
and $\Phi(\Omega_d) \subset E_2$, where $\Phi$ is a translate of $\sum_{j=1}^d (-1)^{d-j+1}\Gamma(t_j)$.  Furthermore, if $(t_1,\ldots,t_d) \in \Omega_d$, $-\frac12 \leq \eta \leq \frac12$, and $i<j$, then
\begin{equation} \label{E:separation lbE}
|t_j-t_i| \gtrsim 
\begin{cases} 
w(a)^{-1} \alpha_1^{\frac12+\eta}\alpha_2^{\frac12-\eta}e^{\frac{A}{d(d+1)}(t_d+t_i)-\frac{2\eta A}{d(d+1)}(t_d-t_i)}, \quad \text{if $j=d$,}&\\
w(a)^{-1} \beta_1 e^{\frac{A}{d(d+1)}(t_j+t_i)}, \quad \text{if $i+1=j \equiv d-1\: (\mathop{mod}2)$,}&\\
w(a)^{-1} \alpha_1 e^{\frac{A}{d(d+1)}(t_j+t_i)}, \quad \text{otherwise.}&
\end{cases}
\end{equation}
The implicit constants depend only on $d$.
\end{lemma}

\begin{lemma} \label{L:lbF}
Under the hypotheses of Lemma~\ref{L:lbF0}, there exists a Borel set $\Omega_d \subset [j_1,j_1+1]^{d-1}\times[j_2,j_2+1]$ with 
$$
\idotsint \chi_{\Omega_d}(t_1,\cdots,t_d) \lambda_\Gamma(t_1)\cdots \lambda_\Gamma(t_d) \, dt_1\, \cdots \, dt_d \gtrsim \beta_1^{\lceil \frac{d}2 \rceil-1}\beta_2 \alpha_1^{\lfloor\frac{d}2\rfloor}
$$
and $\Phi(\Omega_d) \subset F_2$, where $\Phi$ is  translate of $\sum_{j=1}^d (-1)^{d-j} \Gamma(t_j)$.  Furthermore, if $(t_1,\ldots,t_d) \in \Omega_d$, $-\frac12 \leq \eta \leq \frac12$, and $i<j$, then
\begin{equation} \label{E:separation lbF}
|t_j-t_i| \gtrsim 
\begin{cases}
w(a)^{-1} \beta_1^{\frac12+\eta} \beta_2^{\frac12-\eta} e^{\frac{A}{d(d+1)}(t_d+t_i)}e^{-\frac{2\eta A}{d(d+1)}(t_d-t_i)}, \:\: \text{$i+1=j=d$,}\\
w(a)^{-1} (\alpha_1\alpha_2)^{\frac12}e^{\frac{A}{d(d+1)}(t_d+t_i)}, \:\: \text{$i+1<j=d$, $|j_1-j_2| \geq 2$,}\\
w(a)^{-1}\beta_1 e^{\frac{A}{d(d+1)}(t_j+t_i)}, \quad \text{$i+1=j<d$, $j \equiv d\: (\mathop{mod}2)$,}\\
w(a)^{-1}\alpha_1 e^{\frac{A}{d(d+1)}(t_j+t_i)}, \quad \text{otherwise.}
\end{cases}
\end{equation}
Again, the implicit constants depend only on $d$.
\end{lemma}

We now prove Lemmas~\ref{L:lbE0} and~\ref{L:lbF0}, assuming Lemmas~\ref{L:lbE} and~\ref{L:lbF}.  The form of the lower bounds in \eqref{E:separation lbE} and \eqref{E:separation lbF} (which is made possible by ordering the $t_i$ in the proofs of the lemmas) allows us to give an extremely short proof compared with e.g.\ \cite{DLW, BSlms, BSjfa}.  

\begin{proof}[Proofs of Lemmas~\ref{L:lbE0} and~\ref{L:lbF0} from Lemmas~\ref{L:lbE} and~\ref{L:lbF}]
The argument is fairly standard, so we will be brief.  For \eqref{E:lbE},
\begin{equation} \label{E:lb E2}
\begin{aligned}
|E_2| &\gtrsim \int_{\Omega_d} J_\Gamma(t_1,\ldots,t_d)\, dt\\
&\gtrsim \int_{\Omega_d} \prod_{i=1}^d|L_\Gamma(t_i)|^{\frac 1d} \prod_{1 \leq i < j \leq d}( |t_i-t_j| e^{c_d(a_d-a_1)|t_i-t_j|})\, dt
\end{aligned}
\end{equation}
Furthermore, on $\Omega_d$, 
\begin{align}\label{E:vand E}
&\prod_{1 \leq i < j \leq d} |t_j-t_i| \\\notag
&\gtrsim w(a)^{-\frac{d(d-1)}2}e^{\frac{2A(d-1)}{d(d+1)}(t_1+\cdots+t_d)}
\alpha_1^{\frac{d(d-1)}2}\bigl(\tfrac{\beta_1}{\alpha_1}\bigr)^{\lceil \frac d2\rceil-1}\bigl(\tfrac{\alpha_2}{\alpha_1}\bigr)^{(d-1)(\frac12-\eta)}\prod_{i=1}^{d-1}e^{-\frac{2\eta A}{d(d-1)}(t_d-t_i)}\\
\notag
& \gtrsim \prod_{j=1}^d |L_\Gamma(t_j)|^{-\frac{d-1}{d(d+1)}}
\alpha_1^{\frac{d(d-1)}2}\bigl(\tfrac{\beta_1}{\alpha_1}\bigr)^{\lceil \frac d2\rceil-1}\bigl(\tfrac{\alpha_2}{\alpha_1}\bigr)^{(d-1)(\frac12-\eta)}\prod_{1 \leq i < j \leq d}e^{-\frac{2|\eta A|}{d(d-1)}|t_j-t_i|}.
\end{align}
Thus by taking $\eta = C_d^{-1}\min\{1,\tfrac{a_d-a_1}{|A|}\}$, the loss in \eqref{E:vand E} is compensated for by the gain in \eqref{E:lb E2}.  Since
$$
|L_\Gamma(t)|^{\frac 1d-\frac{d-1}{d(d+1)}} = |L_\Gamma(t)|^{\frac 2{d(d+1)}} = \lambda_\Gamma(t),
$$
\eqref{E:lbE} follows by \eqref{E:lbE Omegad} and some arithmetic.

The deductions of \eqref{E:lbF same scale} and \eqref{E:lbF diff scales} from Lemma~\ref{L:lbF} are essentially the same, with the small exception that if $|j_1-j_2| \geq 2$, or if $|j_1-j_2| \leq 1$ and $d \geq 4$, then we use \eqref{E:separation lbF} with $\eta=0$.  We omit the details.
\end{proof}

Now we give the proofs of Lemmas~\ref{L:lbE} and~\ref{L:lbF}.  In both cases, we only give the full details when the dimension is even; the odd dimensional case is similar.  

\begin{proof}[Proof of Lemma~\ref{L:lbE}]  
As in the proof of the restricted weak type inequality, we manage to avoid the ``band structure'' argument entirely by ordering.  Matters are more delicate for the strong type bounds because of the presence of the set $E_2$, but we nonetheless arrive at a vastly shorter proof than that in e.g.\ \cite{BSjfa}.  

To simplify the notation somewhat, we define
\begin{equation} \label{E:alpha beta tilde}
\tilde \alpha_i = w(a)^{-1}e^{\frac{2A}{d(d+1)}j_i}\alpha_i, \qquad \tilde \beta_i = w(a)^{-1}e^{\frac{2A}{d(d+1)}j_i}\beta_i, \qquad i=1,2.
\end{equation}
Then for $i=1,2$, 
\begin{gather} \label{E:geq tilde alpha}
x \in F \implies \int_{j_i}^{j_i+1} \chi_{E_i}(x-\Gamma(t))\, dt \geq \tilde \alpha_i, \\
\label{E:geq tilde beta}
|E_i|^{-1}\int_{E_i} \int_{j_i}^{j_i+1} \chi_F(y+\Gamma(t))\, dt\, dy \geq \tilde \beta_i,\\
\label{E: tildes leq 1}
\tilde \alpha_i, \tilde \beta_i \leq 1.
\end{gather}
Furthermore, for $(t_1,\ldots,t_d) \in [j_1,j_1+1]^{d-1} \times [j_2,j_2+1]$ and $1 \leq i<j \leq d$, \eqref{E:separation lbE} would follow from 
\begin{equation} \label{E:separation tilde}
|t_j-t_i| \gtrsim \begin{cases}
(\tilde \alpha_1+\tilde \alpha_2), \quad &\text{if $j=d$}\\
\tilde \beta_1, \quad &\text{if $i+1=j \equiv d-1$ $(\rm{mod}\: 2)$}\\
\tilde \alpha_1, \quad &\text{otherwise},
\end{cases}
\end{equation}
since $\tilde\alpha_1+\tilde\alpha_2 \geq (\tilde\alpha_1)^{\frac12+\eta}(\tilde\alpha_2)^{\frac12-\eta}$ for all $\eta\in [-\frac12,\frac12]$.  

As in the proof of Proposition~\ref{P:RWT}, our argument is very much inspired by ideas from \cite{TW}, and matters are much clearer in the double-fibration formulation.  We define
$$
U_i = \{(t,x) \in [j_i,j_i+1] \times \R^d : x \in F, x-\Gamma(t) \in E_i\}, \quad i=1,2.
$$
Given a measurable set $U \subset \R^{d+1}$, we define set-valued functions 
$$
\scriptE_U(x) = \{t : (t,x) \in U\}, \qquad \scriptF_U(y) = \{t : (t,y+\Gamma(t)) \in U\}.
$$
Observe that for every $x,y \in \R^d$,
$$\scriptE_{U_i}(x), \scriptF_{U_i}(y) \subset [j_i,j_i+1].$$ 
Furthermore,
$$\scriptE_{U_i}(x) = \emptyset \ctc{if} x \notin F, \quad \scriptF_{U_i}(y) = \emptyset \ctc{if} y \notin E_i, $$
and
\begin{align*}
|\scriptE_{U_i}(x)| &= \int\chi_{E_i}(x-\Gamma(t))\, dt \ctc{if} x \in F, \\
|\scriptF_{U_i}(y)| &= \int\chi_F(x+\Gamma(t))\, dt, \ctc{if} y \in E_i.
\end{align*}

If $S \subset \R$ is a measurable set with $|S| > 0$, we define
$$
m(S) = \inf\{t : |(t,\infty) \cap S| < \tfrac12|S|\}.
$$
We define
\begin{align*}
F^{1,2} &= \{x \in F : m(\scriptE_{U_1}(x)) \leq m(\scriptE_{U_2}(x))\}, \qquad U^{1,2}_1 = \{(t,x) \in U_1 : x \in F^{1,2}\}\\
F^{2,1} &= \{x \in F : m(\scriptE_{U_2}(x)) \leq m(\scriptE_{U_1}(x))\}, \qquad U^{2,1}_1 = \{(t,x) \in U_1: x \in F^{2,1}\}.
\end{align*}
For example, if $j_1 < j_2$, $F^{1,2}=F$ and $F^{2,1}=\emptyset$, and vice-versa if $j_1 > j_2$.

The functions $x \mapsto m(\scriptE_{U_i}(x))$ are Borel, and hence the $F^{i,j}$ and $U^{i,j}$ are Borel sets.  
Since $U_1 = U^{1,2}_1 \cup U^{2,1}_1$, 
$$
|U^{1,2}_1| \geq \tfrac12|U_1|, \qtq{or} |U^{2,1}_1| \geq \tfrac12|U_1|. 
$$
We consider first the case when $|U^{1,2}_1| \geq \frac12|U_1|$.

We define 
\begin{align}\label{E:def U0}
U^0 &= \{(t,x) \in U_2 : x \in F^{1,2}, \: t \geq m(\scriptE_{U_2}(x))+\tfrac{\tilde \alpha_2}4\}\\
\label{E:def U1}
U^1 &= \{(t,x) \in U^{1,2}_1 : t \leq m(\scriptE_{U_1}(x))-\tfrac{\tilde \alpha_1}4\}.
\end{align}
Then if $(t,x) \in U^1$ and $(t',x) \in U^0$, $t' \geq \tfrac14(\tilde \alpha_1+\tilde \alpha_2)$.  Furthermore, for any $x \in F^{1,2}$,
$$
|\scriptE_{U^0}(x)| \geq \tfrac12|\scriptE_{U_2}(x)| - \tfrac14 \tilde \alpha_2 \geq \tfrac14\tilde \alpha_2.
$$
If $x \in F^{1,2}$, $\scriptE_{U^{1,2}_1}(x) = \scriptE_{U_1}(x)$, so if $(t,x) \in U^{1,2}_1$, 
$$
|\scriptE_{U^1}(x)| \geq \tfrac12|\scriptE_{U_1^{1,2}}(x)| - \tfrac{\tilde \alpha_1}4  \geq \tfrac14 |\scriptE_{U^{1,2}_1}(x)|.
$$
Thus by Fubini, 
$$
|U^1| \geq \tfrac14|U^{1,2}_1| \geq \tfrac18|U_1|.
$$

To continue, we define two refinement procedures for a set $U \subseteq U_1$.  Let
\begin{align*}
U_\scriptE &= \{(t,x) \in U : |\scriptE_U(x)| \geq \tfrac12 \tfrac{|U|}{|F|}\}\\
U_\scriptF &= \{(t,x) \in U : |\scriptF_U(x-\Gamma(t))| \geq \tfrac12 \tfrac{|U|}{|E|}\}.
\end{align*}
Then $U_\scriptE$ and $U_\scriptF$ are measurable sets.  Moreover,
$$
|U_\scriptE| = |U|-|U\setminus U_\scriptE| = |U|-\int_{\{x \in F : |\scriptE_U(x)| < \frac12 \frac{|U|}{|F|}\}}|\scriptE_U(x)|\, dx \geq \tfrac12 |U|.
$$
Similarly, $|U_\scriptF| \geq \tfrac12|U|$.  We refine further, defining
\begin{align*}
\langle U \rangle_\scriptE &= \{(t,x) \in U_\scriptE : t \leq m(\scriptE_U(x)) - \tfrac18 \tfrac{|U|}{|F|}\}\\
\langle U \rangle_\scriptF &= \{(t,x) \in U_\scriptF : t \leq m(\scriptF_U(x-\Gamma(t))) - \tfrac18 \tfrac{|U|}{|E|}\}.
\end{align*}

We claim that $|\langle U \rangle_\scriptE| \geq \tfrac18|U|$ and $|\langle U \rangle_\scriptF| \geq \tfrac18|U|$.  Indeed, the former follows from the fact that if $(t,x) \in U_\scriptE$,
$$
|\scriptE_{\langle U \rangle_\scriptE}(x)| \geq \tfrac12 |\scriptE_U(x)| - \tfrac18 \tfrac{|U|}{|F|} \geq \tfrac14|\scriptE_U(x)| = \tfrac14|\scriptE_{U_\scriptE}(x)|,
$$
and so $|\langle U \rangle_\scriptE| \geq \tfrac14|U_\scriptE|$; that $|\langle U \rangle_\scriptF| \geq \tfrac14|U_\scriptF|$ follows by a similar argument.  

If $(t,x) \in \langle U \rangle_\scriptE$,
$$
|\{s \geq t + \tfrac18\tfrac{|U|}{|F|} : (s,x) \in U \}| \geq |(m(\scriptE_U(x)),\infty) \cap \scriptE_U(x)| = \tfrac12 |\scriptE_U(x)| \geq \tfrac14 \tfrac{|U|}{|F|},
$$
and similarly, if $(t,x) \in \langle U \rangle_\scriptF$,
$$
|\{s \geq t + \tfrac18 \tfrac{|U|}{|E|} : (s,x-\Gamma(t)+\Gamma(s)) \in U \}| \geq \tfrac12 |\scriptF_U(x-\Gamma(t))| \geq \tfrac14 \tfrac{|U|}{|E|}.
$$

With $U^1$ as in \eqref{E:def U1}, for $2 \leq i \leq d$, we define recursively 
$$
U^i = \begin{cases}
\langle U^{i-1}\rangle_\scriptF, \qquad &\text{$i$ even}\\
\langle U^{i-1}\rangle_\scriptE, \qquad &\text{$i$ odd}.
\end{cases}
$$
Then $|U^i| \gtrsim |U^{1,2}_1| \gtrsim |U_1|$ for $1 \leq i \leq d$, and so by \eqref{E:geq tilde alpha} and \eqref{E:geq tilde beta}, if $(t,x) \in U^i$ and $c_d>0$ is sufficiently small,
\begin{align*}
&|\{s \geq t + c_d \tilde \beta_1 : (s,x-\Gamma(t)+\Gamma(s)) \in U^{i-1}\}| \geq c_d \tilde\beta_1, \quad &\text{if $i$ is even,}\\
&|\{s \geq t + c_d \tilde\alpha_1 : (s,x) \in U^{i-1}\}| \geq c_d \tilde\alpha_1, \quad &\text{if $i$ is odd.\:}
\end{align*}
Let $(t_0,x_0) \in U^d$.  For $1 \leq k \leq d$ and $(t_1,\ldots,t_k) \in \R^k$, define
$$
x(t_1,\ldots,t_k) = \begin{cases}
x_0-\sum_{j=1}^k (-1)^j \Gamma(t_j), \qquad &\text{if $k$ is odd},\\
x(t_1,\ldots,t_k), \qquad &\text{if $k$ is even.}
\end{cases}
$$

Recalling that $d$ is even, we define
$$
\Omega_1 = \{t : t \geq t_0, (t,x(t)) \in U^{d-1}\},
$$
and for $2 \leq k \leq d-1$,
$$
\Omega_k = \{(t_1,\ldots,t_k) \in \Omega_{k-1} \times \R : t_k \geq t_{k-1}+c_d\rho_k, (t_k,x(t_1,\ldots,t_k)) \in U^{d-k}\}, 
$$
where $\rho_k$ equals $\tilde\alpha_1$ if $k$ is even and $\tilde\beta_1$ if $k$ is odd.  Finally, we define
$$
\Omega_d = \{(t_1,\ldots,t_d) \in \Omega_{d-1} \times \R : t_d \geq t_{d-1}+c_d(\tilde\alpha_1+\tilde\alpha_2), \: (t_d,x(t_1,\ldots,t_d)) \in U^0\}.
$$

That $\Omega_d \subset [j_1,j_1+1]^{d-1} \times [j_2,j_2+1]$ and $\Phi(\Omega_d) \subset E_2$ follow from the definitions of $U_1,U_2$.  By construction, 
$$|\Omega_d| \gtrsim \tilde\alpha_1^{\frac d2-1}\tilde\alpha_2 \tilde\beta_1^{\frac d2},$$
which, by the definition of the $\tilde \alpha_i$ and $\tilde \beta_i$ implies \eqref{E:lbE Omegad}.

Finally, we must verify \eqref{E:separation tilde}.  Since $t_i \in [j_1,j_1+1]$ if $1 \leq i \leq d-1$ and $t_d \in [j_2,j_2+1]$, the second lower bound is trivial.  The first, third, and fourth lower bounds follow from the fact that
\begin{equation} \label{E:monotone2}
t_1+c_d \tilde \alpha_1 \leq t_2,  \: \ldots, \: t_{d-2} + c_d\tilde\beta_1 \leq t_{d-1}, \: t_{d-1}+c_d(\tilde \alpha_1+\tilde \alpha_2) \leq t_d.
\end{equation}

This completes the proof in the case when $|U^{1,2}_1| \geq \tfrac12|U_1|$.  If instead $|U^{2,1}_1| \geq \tfrac12|U_1|$, we define
\begin{align*}
U^0 &= \{(t,x) \in U_2 : x \in F^{1,2}, t \leq m(\scriptE_{U_2}(x))-\tfrac{\tilde\alpha_2}{4}\}\\
 U^1 &= \{(t,x) \in U_1^{2,1} : t \geq m(\mathcal E_{U_1}(x))+\tfrac{\tilde\alpha_1}4\}.
 \end{align*}
Then a similar argument to that given above yields a set $\Omega_d$ on which the monotonicity noted in \eqref{E:monotone2} is reversed, so
$$
t_1-c_d \tilde \alpha_1 \geq t_2,  \: \ldots, \: t_{d-2} - c_d\tilde\beta_1 \geq t_{d-1}, \: t_{d-1}-c_d(\tilde \alpha_1+\tilde \alpha_2) \geq t_d,
$$
and from that it is easy to show that $\Omega_d$ satisfies the conclusions of the lemma.
\end{proof}

\begin{proof}[Proof of Lemma~\ref{L:lbF}]
Define $\tilde\alpha_j$, $\tilde\beta_j$ as above, that is, 
$$
\tilde\alpha_i = w(a)^{-1}e^{\frac{2A}{d(d+1)}j_i}\alpha_i, \qquad \tilde\beta_i = w(a)^{-1}e^{\frac{2A}{d(d+1)}j_i}\beta_i.
$$
Then for $(t_1,\ldots,t_d) \in [j_1,j_1+1]^{d-1}\times[j_2,j_2+1]$, \eqref{E:separation lbF} would follow from 
\begin{equation} \label{E:separation lbF tilde}
|t_j-t_i| \gtrsim 
\begin{cases}
\tilde\beta_1+\tilde\beta_2, \quad &i+1=j=d\\
\tilde\alpha_1+\tilde\alpha_2, \quad &i+1<j=d, |j_1-j_2| \geq 2,\\
\tilde\beta_1, \quad &i+1=j<d, j \equiv d(\rm{mod}\: 2)\\
\tilde\alpha_1, \quad &\text{otherwise}.
\end{cases}
\end{equation}
The proof would be almost exactly the same as the proof of Lemma~\ref{L:lbE}, with the roles of $\tilde\alpha$ and $\tilde\beta$ reversed, were it not for the second case above, $|t_j-t_i| \gtrsim \tilde\alpha_1+\tilde\alpha_2$, if $i+1<j = d$ and $|j_1-j_2| \geq 2$.  But this lower bound is actually trivial, because if $|j_1-j_2| \geq 2$ and $i+1<j=d$, then $|t_j-t_i| \geq |j_1-j_2|-1 \geq 1 \gtrsim \tilde\alpha_1+\tilde\alpha_2$.  

Assuming that the dimension $d$ is even, we argue as in the proof of Lemma~\ref{L:lbE} to construct our parameter set $\Omega_d \subset [j_1,j_1+1]^{d-1}\times[j_2,j_2+1]$ so that exactly one of the following holds for every $(t_1,\ldots,t_d) \in \Omega_d$:  
\begin{gather*}
t_1+c_d\tilde\beta_1 \leq t_2, t_2+c_d\tilde\alpha_1 \leq t_3,\ldots, t_{d-2}+c_d\tilde\alpha_1 \leq t_{d-1}, t_{d-1}+c_d(\tilde\beta_1+\tilde\beta_2) \leq t_d, \\
\text{or}\\
t_1-c_d\tilde\beta_1 \geq t_2, t_2-c_d\tilde\alpha_1 \geq t_3,\ldots, t_{d-2}-c_d\tilde\alpha_1 \geq t_{d-1}, t_{d-1}-c_d(\tilde\beta_1+\tilde\beta_2) \geq t_d.
\end{gather*}
The lower bound \eqref{E:separation lbF tilde} is immediate, and so the lemma is proved.  
\end{proof}

The remainder of the section will be devoted to the proof of Theorem~\ref{T:convl st}.  

\begin{proof}[Proof of Theorem~\ref{T:convl st}]
The rough outline of our argument follows that of Christ in \cite{ChCCC}, but some substantial modifications are made to deal with the weight and the variability of $\theta$.  Some of these modifications are similar to those made in \cite{DSjfa}, and some are new.

We let $p_\theta = p_d/\theta$, $q_\theta = q_d/\theta$ and begin by proving that $T^\theta$ is of weak type $(p_\theta,q_\theta)$ for each $0 <\theta \leq 1$.  Since $T^\theta$ is a positive operator, it suffices to show that 
$$
\langle T^\theta f,\chi_F\rangle \lesssim_{d,\theta,A,a_d-a_1} |F|^{\frac1{q_\theta'}},
$$
whenever $F$ is a Borel set, $f = \sum_k 2^k \chi_{E_k}$, with the $E_k$ disjoint Borel sets, and $\tfrac12 < \|f\|_{L^{p_\theta}} \leq 1$.  The subscripts denote the dependence of the implicit constants; this will always be as described in the statement of Theorem~\ref{T:convl st}.  

We write
$$
\langle T^\theta f,\chi_F\rangle = \sum_{j,k} 2^k \langle T^\theta_j \chi_{E_k},\chi_F\rangle.
$$
For $\eta \in 2^\Z$, define
$$
\scriptK_\eta = \{k \in \Z : \tfrac12\eta <  2^{kp_\theta}|E_k| \leq \eta\}.
$$
For $k \in \Z$ and $\eps \in 2^\Z$, define
$$
\scriptJ_\eps(k) = \{j : \tfrac12\eps |E_k|^{\frac1{p_\theta}}|F|^{\frac1{q_\theta'}} < \langle T^\theta_j \chi_{E_k},\chi_F\rangle \leq \eps |E_k|^{\frac1{p_\theta}}|F|^{\frac1{q_\theta'}}\}.
$$

We claim that
$$
\langle T^\theta f, \chi_F \rangle = \sum_{0 < \eta \leq 1,} \sum_{0 < \eps \leq C_d}\sum_{k \in \scriptK_\eta} \sum_{j \in \scriptJ_\eps(k)} 2^k \langle T^\theta_j \chi_{E_k}, \chi_F\rangle,
$$
where the outer sums are taken over dyadic values of $\eps,\eta$.  That we may take $\eta \leq 1$ follows from our assumption that $\sum 2^{kp_\theta}|E_k| \leq 1$.  That we may take $\eps \leq C_d$ would follow if we knew that each $T^\theta_j$ was of restricted weak type $(p_\theta,q_\theta)$ with uniform constants.  That $T^1_j$ is of restricted weak type $(p_1,q_1)$ follows from Lemma~\ref{L:lbE0} with $\alpha_1=\alpha_2$, $\beta_1=\beta_2$, and $j_1=j_2=j$, and that $T^0_j$ is of strong type $(\infty,\infty)$ is elementary.  Thus $T^\theta_j$ is indeed of restricted weak type $(p_\theta,q_\theta)$ by interpolation.  

For each pair $(j,k)$, define
\begin{align*}
F^{(j,k)} &= \{x \in F : T^\theta_j\chi_{E_k}(x) \geq \tfrac{\langle T^\theta_j\chi_{E_k},\chi_F\rangle}{2|F|}\}\\
E^{(j,k)} &= \{y \in E_k : (T^\theta_j)^*\chi_{F^{(j,k)}}(y) \geq \tfrac{\langle T^\theta_j \chi_{E_k}, \chi_{F^{(j,k)}}\rangle}{2|E_k|}\}.
\end{align*}
Standard arguments show that 
\begin{equation} \label{E:Ejk Fjk big}
\langle T^\theta_j\chi_{E_k},\chi_F \rangle \geq \langle T^\theta_j \chi_{E^{(j,k)}},\chi_{F^{(j,k)}}\rangle \geq \tfrac14\langle T^\theta_j\chi_{E_k},\chi_F \rangle.
\end{equation}

\begin{lemma}\label{L:bound cardJ}
For all $\eps>0$, and $k \in \Z$,
\begin{equation} \label{E:card J}
\# \scriptJ_{\eps}(k) \lesssim (\log(1+\eps^{-1}))^4\eps^{-\frac1{1+p_\theta^{-1}-q_\theta^{-1}}}.
\end{equation}
Furthermore,
\begin{equation} \label{E:disjoint card J0}
\begin{gathered}
\sum_{j\in \scriptJ_\eps(k)} |E^{(j,k)}| \lesssim (\log(1+\eps^{-1}))^3|E_k|, \\
 \sum_{k \in \scriptK_\eta}\sum_{j\in \scriptJ_\eps(k)} |F^{(j,k)}| \lesssim (\log(1+\eps^{-1}))^4|F|.
 \end{gathered}
\end{equation}
The implicit constants depend only on $d$, $\theta$, and an upper bound for $|a_d-a_1|^{-1}$.  
\end{lemma}

\begin{proof}[Proof of Lemma~\ref{L:bound cardJ}]
Define
$$
\scriptJ_{\eps,\rho_1,\rho_2}(k) = \{j \in \scriptJ_\eps(k) : \tfrac 12 \rho_1|E_k| < |E^{(j,k)}| \leq \rho_1|E_k|,\: \tfrac12 \rho_2|F| < |F^{(j,k)}| \leq \rho_2|F|\}.
$$
If $j \in \scriptJ_{\eps,\rho_1,\rho_2}(k)$, by the restricted weak type bound noted above,
$$
\tfrac 12 \eps |E_k|^{\frac1{p_\theta}}|F|^{\frac1{q_\theta'}} \leq \langle T^\theta_j \chi_{E_k},\chi_{F_k} \rangle \leq 4 \langle T_j^\theta \chi_{E^{(j,k)}},\chi_{F^{(j,k)}} \rangle \leq C_d \rho_1^{\frac1{p_\theta}}\rho_2^{\frac1{q_\theta'}}|E_k|^{\frac1{p_\theta}}|F|^{\frac1{q_\theta'}}.
$$
Thus if $\#\scriptJ_{\eps,\rho_1,\rho_2}(k) \neq 0$, 
\begin{equation} \label{E:compare rho eps}
C_d^{-1} \eps \leq \rho_1^{\frac1{p_\theta}}\rho_2^{\frac1{q_\theta'}} \leq \min\{\rho_1^{\frac1{p_\theta}},\rho_2^{\frac 1{q_\theta'}}\}.
\end{equation}

We claim that if $\scriptJ \subset \scriptJ_{\eps,\rho_1,\rho_2}(k)$ is a finite $\frac{C_{d,\theta}}{|a_d-a_1|}\log(1+ \eps^{-1})$-separated set, then
\begin{equation} \label{E:sum Es}
\sum_{j \in \scriptJ} |E^{(j,k)}| \leq 4  |E_k|, 
\end{equation}
and furthermore that if 
$$
\scriptL \subset \{(j,k) : k \in \scriptK_\eta, j \in \scriptJ_{\eps,\rho_1,\rho_2}(k)\},
$$
is a finite set with the property that $(j_1,k_1),(j_2,k_2) \in \scriptL$ implies $|k_1-k_2| \geq C_{d,\theta}\log(1+\eps^{-1})$ or $k_1=k_2$ and $|j_1-j_2| \geq \tfrac{C_{d,\theta}}{a_d-a_1}\log(1+\eps^{-1})$, then
\begin{equation} \label{E:sum Fs}
\sum_{(j,k) \in \scriptL} |F^{(j,k)}| \leq 4|F|.
\end{equation}
Here $C_{d,\theta}$ is a large constant, to be determined in a few paragraphs.  Inequalities \eqref{E:sum Es} and \eqref{E:sum Fs} certainly imply \eqref{E:disjoint card J0} by summing on dyadic values of $\rho_1,\rho_2$ satisfying \eqref{E:compare rho eps}.  

We start with \eqref{E:sum Es}, which we will prove by contradiction.  We know that 
$$
\tfrac12\#\scriptJ \rho_1|E_k| \leq \sum_{j \in \scriptJ} |E^{(j,k)}| \leq 2 \#\scriptJ \rho_1 |E_k|.
$$
By Cauchy--Schwarz,
\begin{align*}
\sum_{j \in \scriptJ}|E^{(j,k)}| &= \int_{E_k} \sum_{j \in \scriptJ}\chi_{E^{(j,k)}} \leq |E_k|^{\frac12}(\int(\sum_{j \in \scriptJ} \chi_{E^{(j,k)}})^2)^{\frac12} \\
&= |E_k|^{\frac12}(\sum_{j \in \scriptJ}|E^{(j,k)}| + \sum_{j_1 \neq j_2}|E^{(j_1,k)} \cap E^{(j_2,k)}|)^{\frac12},
\end{align*}
so if $\sum_{j \in \scriptJ}|E^{(j,k)}| > 4|E_k|$, 
$$
 (\#\scriptJ \rho_1 |E_k|)^2 \lesssim_d |E_k|\sum_{j_1 \neq j_2} |E^{(j_1,k)} \cap E^{(j_2,k)}| \lesssim_d |E_k| (\#\scriptJ)^2 \sup_{j_1 \neq j_2}|E^{(j_1,k)} \cap E^{(j_2,k)}|.
$$
Hence there exist $j_1 \neq j_2$ such that
$$
\rho_1^2 |E_k| \lesssim_d |E^{(j_1,k)} \cap E^{(j_2,k)}|.
$$

Assuming the conclusion of the previous paragraph, let $G = E^{(j_1,k)} \cap E^{(j_2,k)}$.  Since $\lambda_\Gamma(t)$ is almost constant on each interval $[j,j+1]$, $T_j^\theta \chi_E \sim_d \lambda_\Gamma(j)^\theta T_j^0\chi_E$, for any set $E$.  Thus
\begin{align*}
 \langle T_j^\theta \chi_{E_k},\chi_F \rangle \sim_d (\langle T_j^1 \chi_{E_k},\chi_F\rangle)^\theta(\langle T_j^0\chi_{E_k},\chi_F\rangle)^{1-\theta} 
\lesssim_d (\langle T_j^1 \chi_{E_k},\chi_F\rangle)^\theta |F|^{1-\theta}.
\end{align*}
Thus if $j \in \scriptJ_\eps(k)$, with a little more arithmetic, we arrive at
$$
\langle T_j^1\chi_{E_k},\chi_F \rangle \gtrsim_{d,\theta} \eps^{\frac1\theta}|E_k|^{\frac1{p_1}}|F|^{\frac1{q_1'}}.
$$
By the containment $F^{(j,k)} \subseteq F$, the definition of $E^{(j,k)}$, the near constancy of $\lambda_\Gamma$ on $[j,j+1]$, and the previous observations, for $y \in G$
$$
(T_{j_i}^1)^*\chi_F(y) \gtrsim_{d,\theta}  \tfrac{\langle T_{j_i}^1\chi_{E_k},\chi_{F^{(j,k)}}\rangle}{|E_k|} \gtrsim_{d,\theta} \tfrac{\langle T_{j_i}^1\chi_{E_k},\chi_{F}\rangle}{|E_k|} \gtrsim_{d,\theta} \eps^{\frac1\theta}\frac{|F|^{\frac1{q_1'}}}{|E_k|^{\frac1{p_1'}}} =: \beta_i.
$$
Additionally,
\begin{align*}
\alpha_i &:= \frac{\langle T_{j_i}^1 \chi_G, \chi_F \rangle}{|F|} \gtrsim_{d,\theta} \frac{\beta_i|G|}{|F|} = \eps^{\frac1\theta} \frac{|G|}{|F|^{\frac1{q_1}}|E_k|^{\frac1{p_1'}}} \\
&\gtrsim_{d,\theta} \eps^{\frac1\theta}\rho_1^2 \frac{|E_k|^{\frac1{p_1}}}{|F|^{\frac1{q_1}}} 
\gtrsim_{d,\theta} \eps^{\frac1\theta + 2p_\theta}\frac{|E_k|^{\frac1{p_1}}}{|F|^{\frac1{q_1}}}.
\end{align*}
By Lemma~\ref{L:lbF0} and a little algebra, if $|j_1-j_2| \geq 2$, 
$$
1 \gtrsim_{d,\theta} \eps^{\frac{d(d+1)}{2\theta}+d(d-1)p_\theta} e^{c_d(a_d-a_1)|j_1-j_2|}.
$$
Since $a_d > a_1$ by assumption, this gives a contradiction if the constant $C$ in the separation condition above \eqref{E:sum Es} is larger than a dimensional constant times $\tfrac1{|a_d-a_1|}$.  Thus we must have that $\sum_{j \in \scriptJ}|E^{(j,k)}| \leq 4 |E_k|$.  

Now we turn to \eqref{E:sum Fs}.  Arguing as before, if $\sum_{(j,k) \in \scriptL} |F^{(j,k)}| > 4 |F|$, there exist $(j_1,k_1) \neq (j_2,k_2)$ such that 
$$
\rho_2^2|F| \lesssim_d |F^{(j_1,k_1)}\cap F^{(j_2,k_2)}|.
$$
Let $H = F^{(j_1,k_1)}\cap F^{(j_2,k_2)}$.  Arguing in a similar manner as before, for $x \in H$,
$$
T_{j_i}^1 \chi_{E_{k_i}}(x) \gtrsim_{d,\theta} \eps^{\frac 1\theta}\frac{|E_{k_i}|^{\frac 1{p_1}}}{|F|^{\frac 1{q_1}}} =: \alpha_i
$$
and
$$
\beta_i := \frac{\langle T_{j_i}^1\chi_{E_{k_i}},\chi_G\rangle}{|E_{k_i}|} \gtrsim_{d,\theta}\eps^{\frac 1\theta+2q_\theta'}\frac{|F|^{\frac1{q_1'}}}{|E_{k_i}|^{\frac 1{p_1'}}}.
$$
Applying Lemma~\ref{L:lbE0} this time, and performing the necessary arithmetic,
\begin{gather*}
|E_{k_2}|^{\frac\delta{p_1}} \gtrsim_{d,\theta} e^{(a_d-a_1)|j_1-j_2|} \eps^{\frac{d(d+1)}{2\theta}+2(d-1)q_\theta'}|E_{k_1}|^{\frac \delta{p_1}}\\
|E_{k_1}|^{\frac\delta{p_1}} \gtrsim_{d,\theta} e^{(a_d-a_1)|j_1-j_2|} \eps^{\frac{d(d+1)}{2\theta}+2(d-1)q_\theta'}|E_{k_2}|^{\frac \delta{p_1}}.
\end{gather*}
If $|k_1-k_2| \geq C_{d,\theta}\log(1+\eps^{-1})$, we ignore the exponential terms (which only help us find a contradiction), and derive a contradiction from the fact that $|E_{k_j}| \sim \eta 2^{-k_j p_\theta}$.  Therefore $k_1=k_2$ and $|j_1-j_2| \geq \tfrac{C_{d,\theta}}{a_d-a_1}$, so
$$
1 \gtrsim_{d,\theta} e^{(a_d-a_1)|j_1-j_2|} \eps^{\frac 1\theta(\frac{d(d+1)}2+2(d-1)q_\theta'},
$$
and again, we arrive at a contradiction for $C_{d,\theta}$ sufficiently large, so \eqref{E:sum Fs} must hold.  

Finally, we prove \eqref{E:card J}.  Inequalities \eqref{E:sum Es} and \eqref{E:sum Fs} imply that
\begin{equation} \label{E:disjoint card J}
\begin{gathered}
\sum_{j \in \scriptJ_{\eps,\rho_1,\rho_2}(k)}|E^{(j,k)}| \lesssim_{d,\theta} \tfrac{1}{a_d-a_1} \log(1+\eps^{-1})|E_k|, \\
 \sum_{j \in \scriptJ_{\eps,\rho_1,\rho_2}(k)} |F^{(j,k)}| \lesssim_{d,\theta} \tfrac{1}{a_d-a_1} \log(1+\eps^{-1})^2|F|.
 \end{gathered}
\end{equation}
This implies that $(\rho_1+\rho_2)\#\scriptJ_{\eps,\rho_1,\rho_2}(k) \lesssim_{d,\theta} \tfrac{1}{a_d-a_1} \log(1+\eps^{-1})^2$.  Since
$$
2 \geq (\rho_1+\rho_2) \geq (\rho_1^{\frac1{p_\theta}}\rho_2^{\frac1{q_\theta'}})^{\frac1{1+p_\theta^{-1}-q_\theta^{-1}}}\gtrsim_{d,\theta} \eps^{\frac1{1+p_\theta^{-1}-q_\theta^{-1}}},
$$
this implies that 
\begin{equation} \label{E:card J'}
\#\scriptJ_{\eps,\rho_1,\rho_2}(k) \lesssim_{d,\theta} \tfrac1{a_d-a_1} \log(1+ \eps^{-1})^2\eps^{-\frac1{1+p_\theta^{-1}-q_\theta^{-1}}},
\end{equation}
and, in light of \eqref{E:compare rho eps}, inequality \eqref{E:card J} follows by summing \eqref{E:card J'} over dyadic values of $\rho_1$ and $\rho_2$.  

This completes the proof of Lemma~\ref{L:bound cardJ}.
\end{proof}

Now we finish proving the weak type bound.  Fix $\eta,\eps$.  By the definitions of $\scriptJ_\eps(k)$ and $\scriptK_\eta$, \eqref{E:card J}, and the trivial estimate $\#\scriptK_\eta \leq 2\eta^{-1}$,
\begin{align*}
&\sum_{k \in \scriptK_\eta}\sum_{j \in \scriptJ_\eps(k)}2^k \langle T_j^\theta \chi_{E_k},\chi_F\rangle \sim_d \sum_{k \in \scriptK_\eta}\sum_{j \in \scriptJ_\eps(k)} \sim \eps \eta^{\frac 1{p_\theta}} |F|^{\frac 1{q_\theta'}}\\
&\qquad \lesssim_{d,\theta,a_d-a_1} \sum_{k \in \scriptK_\eta}(\log(1+\eps^{-1}))^3 \eps^{\frac{p_\theta^{-1}-q_\theta^{-1}}{1+p_\theta^{-1}-q_\theta^{-1}}} \eta^{\frac 1{p_\theta}}|F|^{\frac 1{q_\theta'}}\\
&\qquad \lesssim_{d,\theta,a_d-a_1} (\log(1+\eps^{-1}))^3 \eps^{\frac{p_\theta^{-1}-q_\theta^{-1}}{1+p_\theta^{-1}-q_\theta^{-1}}}\eta^{-\frac1{p_\theta'}}|F|^{\frac1{q_\theta'}}.
\end{align*}
Next, by \eqref{E:Ejk Fjk big} and the restricted weak type bound, H\"older's inequality, inequality \eqref{E:disjoint card J0}, H\"older's inequality and the fact that $p_\theta<q_\theta$, and the definition of $\scriptK_\eta$,
\begin{equation}\label{E:pos eta}
\begin{aligned}
&\sum_{k \in \scriptK_\eta}\sum_{j \in \scriptJ_\eps(k)} 2^k \langle T_j^\theta \chi_{E_k}, \chi_F \rangle
\lesssim_d \sum_{k \in \scriptK_\eta}\sum_{j \in \scriptJ_\eps(k)}2^k |E^{(j,k)}|^{\frac1{p_\theta}}|F^{(j,k)}|^{\frac1{q_\theta'}}\\
&\qquad \lesssim_d \bigl(\sum_{k \in \scriptK_\eta}\sum_{j \in \scriptJ_\eps(k)} 2^{kq_\theta}|E^{(j,k)}|^{\frac{q_\theta}{p_\theta}}\bigr)^{\frac1{q_\theta}}\bigl(\sum_{k \in \scriptK_\eta}\sum_{j \in \scriptJ_\eps(k)}|F^{(j,k)}|\bigr)^{\frac1{q_\theta'}}\\
&\qquad \lesssim_{d,\theta,a_d-a_1} (\log(1+\eps^{-1}))^{c_{d,\theta}}
\sup_{k \in \scriptK_\eta, j \in \scriptJ_\eps(k)}(2^{kp_\theta}|E_k|)^{\frac1{p_\theta}-\frac1{q_\theta}}\\
&\qquad\qquad\qquad \qquad \times \bigl(\sum_{k \in \scriptK_\eta}2^{kp_\theta}|E_k|\bigr)^{\frac1{q_\theta}}|F|^{\frac1{q_\theta'}}\\
&\qquad \lesssim_{d,\theta,a_d-a_1} (\log(1+\eps^{-1}))^{c_{d,\theta}} \eta^{\frac1{p_\theta}-\frac1{q_\theta}}|F|^{\frac1{q_\theta'}}.
\end{aligned}
\end{equation}

Combining the previous two estimates,
\begin{equation} \label{E:wt precursor}
\sum_{k \in \scriptK_\eta}\sum_{j \in \scriptJ_\eps(k)} 2^k \langle T_j^\theta \chi_{E_k},\chi_F\rangle \lesssim \eps^a \eta^b |F|^{\frac1{q_\theta'}},
\end{equation}
for some $a=a_{d,\theta}>0,b=b_{d,\theta}>0$.  Summing on dyadic values of $\eta \leq 1$ and $\eps \lesssim_{d,\theta} 1$ gives the weak type bound.  

Now we turn to the strong type bound.  The argument is similar, so we simply sketch it.  It suffices to prove that $\langle T^\theta f,g\rangle \lesssim_{d,\theta,a_d-a_1} 1$ when $f = \sum_k 2^k \chi_{E_k}$ and $g = \sum_l 2^l \chi_{F_l}$, where the $E_k$, and likewise the $F_l$, are pairwise disjoint Borel sets and $\tfrac12 < \|f\|_{L^{p_\theta}} \leq 1$, $\tfrac12 < \|g\|_{L^{q_\theta'}} \leq 1$.  

For $\eta_1,\eta_2,\eps>0$, we define $\scriptK_{\eta_2}$ as above, and define
\begin{gather*}
\scriptL_{\eta_1} = \{l \in \Z : \eta_1 < 2^{lq_\theta'}|F_l| \leq \eta_1\}, \\
 \scriptJ_\eps(k,l) := \{j \in \Z : \tfrac12\eps |E_k|^{\frac1{p_\theta}}|F_l|^{\frac1{q_\theta'}} < \langle T_j^\theta \chi_{E_k},\chi_{F_l}\rangle \leq \eps |E_k|^{\frac1{p_\theta}}|F_l|^{\frac1{q_\theta'}}\}.
 \end{gather*}
Then
$$
\langle Tf,g\rangle = \sum_{\eta_1, \eta_2,\eps} \sum_{l \in \scriptL_{\eta_1}}\sum_{k \in \scriptK_{\eta_2}}\sum_{j \in \scriptJ_\eps(k,l)} 2^{k+l} \langle T_j^\theta \chi_{E_k}, \chi_{F_l}\rangle,
$$
where the sum is taken over dyadic values of $\eta_1,\eta_2,\eps$ with $0 < \eta_1,\eta_2 \leq 1$ and $0 < \eps \lesssim_{d,\theta} 1$.  

By \eqref{E:wt precursor}, then the definition of $\scriptL_{\eta_1}$ combined with the trivial bound $\#\scriptL_{\eta_1} \lesssim_d \eta_1^{-1}$,
\begin{align*}
\sum_{l \in \scriptL_{\eta_1}}\sum_{k \in \scriptK_{\eta_2}}\sum_{j \in \scriptJ_\eps(k,l)} 2^{k+l} \langle T_j^\theta \chi_{E_k},\chi_{F_l}\rangle 
&\lesssim_{d,\theta,a_d-a_1} \sum_{l \in \scriptL_{\eta_1}} \eps^a \eta_1^b 2^l |F_l|^{\frac1{q_\theta'}} \\
&\lesssim_{d,\theta,a_d-a_1} \eps^a \eta_1^{-\frac1{q_\theta'}}\eta_2^b.
\end{align*}

Now we seek a bound with a positive power of $\eta_1$.  Define
\begin{align*}
F^{(j,k,l)} &= \{x \in F_l : T_j^\theta \chi_{E_k}(x) \geq \tfrac12 \tfrac{\langle T_j^\theta \chi_{E_k},\chi_{F_l}\rangle}{|F_l|}\}\\
E^{(j,k,l)} &= \{y \in E_k : (T_j^\theta)^* \chi_{F^{(j,k,l)}}(y) \geq \tfrac12 \tfrac{\langle T_j^\theta \chi_{E_k}, \chi_{F^{(j,k,l)}}\rangle}{|E_k|}\}.
\end{align*}
By Lemma~\ref{L:bound cardJ}, for each $l$,
$$
\sum_{k \in \scriptK_{\eta_1}}\sum_{j \in \scriptJ_\eps(k,l)} |F^{(j,k,l)}| \lesssim_{d,\theta,a_d-a_1} (\log(1+\eps^{-1}))^4|F_l|.
$$
Similar arguments show that
$$
\sum_{l \in \scriptL_{\eta_2}}\sum_{j \in \scriptJ_\eps(k,l)}|E^{(j,k,l)}| \lesssim_{d,\theta,a_d-a_1} (\log(1+\eps^{-1}))^4 |E_k|.
$$
Arguing similarly to \eqref{E:pos eta},
\begin{align*}
&\sum_{l \in \scriptL_{\eta_1}}\sum_{k \in \scriptK_{\eta_2}}\sum_{j \in \scriptJ_{\eps}(k,l)} 2^{k+l} \langle T_j^\theta \chi_{E_k},\chi_{F_l}\rangle \\
&\qquad \lesssim_{d,\theta} \bigl(\sum_{l \in \scriptL_{\eta_1}}\sum_{k \in \scriptK_{\eta_2}}\sum_{j \in \scriptJ_\eps(k,l)} 2^{lp_\theta'}|F^{(j,k,l)}|^{\frac{p_\theta'}{q_\theta'}}\bigr)^{\frac1{p_\theta'}} \\
& \qquad \qquad \qquad \times \bigl(\sum_{k \in \scriptK_{\eta_2}}\sum_{l \in \scriptL_{\eta_1}}\sum_{j \in \scriptJ_\eps(k,l)} 2^{kp_\theta}|E^{(j,k,l)}|\bigr)^{\frac1{p_\theta}}\\
&\qquad \lesssim_{d,\theta,a_d-a_1} (\log(1+\eps^{-1}))^{c_{d,\theta}}\eta_1^{\frac1{q_\theta'}-\frac1{p_\theta'}},
\end{align*}
which has the positive power of $\eta_1$ that we wanted.  

Combining our two upper bounds, 
$$
\sum_{l \in \scriptL_{\eta_1}}\sum_{k \in \scriptK_{\eta_2}}\sum_{j \in \scriptJ_{\eps}(k,l)} 2^{k+l} \langle T_j^\theta \chi_{E_k},\chi_{F_l}\rangle \lesssim_{d,\theta,a_d-a_1} \eps^a \eta_1^b\eta_2^c,
$$
for constants $a=a_{d,\theta}>0$, $b=b_{d,\theta}>0$, $c=c_{d,\theta}>0$.  Summing on dyadic values of $\eta_1,\eta_2,\eps$ gives the strong type bound.  
\end{proof}

\subsection*{Special case:  polynomial curves}

In \cite{DLW, oberlin-polynomial, BSjfa}, uniform endpoint strong type bounds are obtained for convolution with affine arclength measure on polynomial curves.  We can use the techniques of this section together with a geometric inequality for polynomial curves from \cite{DW} to prove uniform fractional integral analogues.  

\begin{proposition} \label{frac-poly}
Let $\gamma:\R\to\R^d$ be a polynomial curve of degree $N$ and let $Z_\gamma$ be any finite set containing the real parts of the complex zeroes of $L_\gamma$. Then the operator
$$
T^\theta_\gamma f(x) = \int_\R f(x-\gamma(t))\, \lambda_\gamma(t)^{\theta} \, \tfrac{dt}{\tnorm{dist}(t,Z_\gamma)^{1-\theta}}
$$
satisfies
\begin{equation} \label{E: conv poly}
\|T^\theta_\gamma f\|_{L^q(\R^d)}\lesssim  \|f\|_{L^p(\R^d)}, 
\end{equation}
for all $0 < \theta \leq 1$ and $(p^{-1},q^{-1})$ on the line segment with endpoints $(\tfrac\theta{p_d},\tfrac\theta{q_d})$ and $(1-\tfrac\theta{p_d},1-\tfrac\theta{q_d})$.  The implicit constant in \eqref{E: conv poly} depends only on $d$, $\theta$, and $N$, and $\#Z_\gamma$.  
\end{proposition}

A related estimate will also appear in \cite{BSpolyRest}.

\begin{proof}
It was proved in \cite{DW} that we may decompose $\R$ as a union of intervals $\R = \bigcup_{j=1}^{C_{d,N}}$ in such a way that for each $j$ and $(t_1,\ldots,t_d) \in I_j^d$,
$$
|J_\gamma(t_1,\ldots,t_d)| \gtrsim \prod_{j=1}^d|L_\gamma(t_j)|^{\frac1d} \prod_{i<j} |t_i-t_j|,
$$
with implicit constants depending only on $d,N$.  By the triangle inequality, it suffices to prove the estimates for $\gamma|_I$ for a single interval $I=I_j$ as above.  We will use the following simple and widely-known polynomial lemma (cf.\ \cite{DW}).

\begin{lemma}  We may decompose $I = \bigcup_{j=1}^{C(\#Z_\gamma)} I_j$ so that for $t \in I_j$, $\dist(t,Z_\gamma) = |t-a(I_j)|$ and $|L_\gamma(t)| \sim C_\gamma (I_j) |t-a(I_j)|^{k(I_j)}$.  Here $a(I_j) \in Z_\gamma$ and $0 \leq k(I_j) \leq \deg L_\gamma$, and the implicit constants depend only on the degree of $L_\gamma$.  
\end{lemma}

The proof is very short, so we include it here.

\begin{proof}
Making an initial decomposition if necessary, we may assume that $\dist(t,Z_\gamma) \linebreak = |t-a_0|$, $t \in I$.  Translating if needed, we may assume that $a_0=0$, and by reordering, $Z_\gamma = \{a_0,a_1,\ldots,a_M\}$, where $0=|a_0| \leq |a_1| \leq \cdots \leq |a_M|$.  For convenience, set $a_{M+1}=\infty$.  Define
$$
A_j = \{t \in I : \tfrac12|a_j| \leq |t| \leq \tfrac12|a_{j+1}|\}, \qquad 0 \leq j \leq M,
$$
and observe that $I = \bigcup A_j$.  

Repeating some of the $a_j$ if necessary, we may write $P(t) = C \prod_{j=1}^M (t-z_j)^{k_j}$, where $a_j = \rm{Re}\, z_j$ and the $k_j$ are allowed to be zero.  Set $b_j = \rm{Im}\, z_j$.  Since $|t-z_j| \sim |t-a_j| + |b_j|$, $|P(t)| \sim \sum_{i=1}^{C(\#Z_\gamma)} C_i \prod_{j=1}^M|t-a_j|^{k_{j,i}}$.  Decomposing further if needed, we may assume that a single one of the polynomials on the right is largest on $I$, $P(t) \sim C \prod_{j=1}^M|t-a_j|^{k_j}$.  But then on $A_j$,
$$
|P(t)| \sim C \prod_{i=1}^j |t|^{k_i} \prod_{i=j+1}^M |a_i|^{k_i}.
$$
\end{proof}

It now suffices to prove the estimates in the proposition on a single one of the intervals from the lemma, which we also denote $I$.  Translating and reflecting if necessary, we may assume that $b(I) = 0$ and $I \subseteq [0,\infty)$.  In summary, on $I$, 
$$
|L_\gamma(t)| \sim C_\gamma t^k, \qquad |J_\gamma(t_1,\ldots,t_d)| \gtrsim C_\gamma \prod_{j=1}^d |t_j|^{k/d}\prod_{i<j} |t_i-t_j|.
$$
Reparametrizing so that $\Gamma(t) = \gamma(e^{-t})$, these estimates immediately imply that on $J:= e^{-I}$, 
$$
|L_\Gamma(t)| \sim C_\gamma e^{-(k+\tfrac{d^2+d}2)t}, \quad |J_\Gamma(t_1,\ldots,t_d)| \gtrsim C_\gamma \exp[-(\tfrac kd+d)\sum_i t_i] \prod_{i<j} |e^{-t_i}-e^{-t_j}|.
$$
Using the elementary estimate 
$$
|e^{-t_i}-e^{-t_j}| \gtrsim \exp[-\tfrac{t_i+t_j}2]\exp[\tfrac14|t_i-t_j|]|t_i-t_j|,
$$
we obtain the lower bound in Proposition~\ref{P:convl exp gi}, except without the gain of $a_d-a_1$ in the exponent.  Nevertheless, this is sufficient to apply the techniques of this section and obtain the proposition.
\end{proof}

One could also obtain analogues of the above proposition and lemma for $\gamma$ a rational curve. 

Proposition~\ref{frac-poly} gives a simple result bounding the unweighted operator
$$
S_\gamma f(x) = \int_\R f(x-\gamma(t))\, dt.
$$
Let $N_{\rm{loc}}$ equal the maximum order of vanishing of $L_\gamma$ on $\R$ and $N_{\rm{glob}}$ equal the degree of $L_\gamma$.  Let $\theta_\bullet = (1+\tfrac{2N_\bullet}{d(d+1)})^{-1}$.  Observe that $\theta_{\rm{glob}} \leq \theta_{\rm{loc}}$ and that if $\theta_{\rm{glob}} \leq \theta \leq \theta_{\rm{loc}}$, 
$$
|\lambda_\gamma(t)|^\theta \dist(t,Z_\gamma)^{\theta-1} \gtrsim_\gamma 1.
$$
Thus we obtain the following.

\begin{corollary}  Let $\gamma:\R \to \R^d$ be a polynomial curve of degree $N$.  The unweighted operator
$$
\|S_\gamma f\|_{L^q(\R^d)} \lesssim_\gamma \|f\|_{L^p(\R^d)},
$$
for all $(p^{-1},q^{-1})$ in the convex hull of the points
$$
(\tfrac{\theta_{\rm{loc}}}{p_d}, \tfrac{\theta_{\rm{loc}}}{q_d}), \quad (\tfrac{\theta_{\rm{glob}}}{p_d}, \tfrac{\theta_{\rm{glob}}}{q_d}), \quad (1-\tfrac{\theta_{\rm{loc}}}{q_d}, 1-\tfrac{\theta_{\rm{loc}}}{p_d}), \quad  (1-\tfrac{\theta_{\rm{glob}}}{q_d}, 1-\tfrac{\theta_{\rm{glob}}}{p_d}).
$$
\end{corollary}
The dependence of the implicit constant on $\gamma$ is unavoidable because of the lack of scale invariance.  This is the sharp Lebesgue space result.  One could use a similar argument obtain optimal Lebesgue space bounds for convolution with Euclidean arclength on $\gamma$, but the exponents would not be quite as simple.  

\section{Proofs of the geometric inequalities and other loose ends} \label{S:gi proofs}

Finally, we give the promised proofs of Propositions~\ref{P:convl gi}, \ref{P:jacobian bounds}, \ref{P:one to one}, and~\ref{P:convl exp gi}.  

We begin with Proposition~\ref{P:convl gi}.  Recalling that $L_\gamma^j := L_{(\gamma_1,\ldots,\gamma_j)}$ for $1 \leq j \leq d$ and that $L_\gamma^0\equiv L_\gamma^{-1} \equiv 1$, we define
$$
A_\gamma^k := \frac{L_\gamma^{d-k-1}L_\gamma^{d-k+1}}{(L_\gamma^{d-k})^2}.
$$
We note that the quantities whose almost log convexity and monotonicity we had assumed are
\begin{equation} \label{E:A to log guy}
\frac{L_\gamma^d (L_\gamma^{d-k-1})^k}{(L_\gamma^{d-k})^{k+1}} = \prod_{j=1}^k (A_\gamma^j)^j.
\end{equation}

It was proved in \cite{spyrosthesis} (see also \cite{DW}) that if we define $J_\gamma^1 := A_\gamma^1$ and, for $2\le k \le d$,
\begin{equation} \label{E:Jk}
J_\gamma^k(t_1,\ldots,t_k) := \prod_{i=1}^k A_\gamma^k(t_i) \int_{t_1}^{t_2} \cdots \int_{t_{k-1}}^{t_k} J_\gamma^{k-1}(s_1,\ldots,s_{k-1}) \, ds_{k-1} \cdots ds_1,
\end{equation}
then if each $L_\gamma^j$ is non-vanishing on some interval $I$ and if $(t_1,\ldots,t_d) \in I^d$, we have the identity 
\begin{equation} \label{E:Jd is J}
J_\gamma^d(t_1,\ldots,t_d) = \det(\gamma'(t_1), \ldots,\gamma'(t_d)).
\end{equation}
(The result is only claimed for polynomial curves, but the proof under our hypotheses is unchanged.) 

We want to establish the geometric inequality 
$$
|J_\gamma^d(t_1,\ldots,t_d)| \gtrsim \prod_{j=1}^d |L_\gamma^d(t_j)|^{\frac1d}\prod_{i<j}|t_j-t_i|;
$$
this is just the case $k=d$ of the following lemma.  

\begin{lemma} \label{L:convl gi} 
Under the hypotheses of Proposition~\ref{P:convl gi}, if $1 \leq k \leq d$ and $(t_1,\ldots,t_k) \linebreak \in I^k$ satisfies $t_1 < \cdots < t_k$, then
\begin{equation} \label{E:convl gi k}
|J_\gamma^k(t_1,\ldots,t_k)| \gtrsim \prod_{i=1}^k \prod_{j=1}^k |A_\gamma^j(t_i)|^{\frac jk} \prod_{1 \leq i < j \leq k} |t_j-t_i|.
\end{equation}
\end{lemma}

\begin{proof}[Proof of Lemma~\ref{L:convl gi}] 
Multiplying appropriate coordinates of $\gamma$ by $-1$ if necessary, we may assume that the $L^k_\gamma$, and hence the $A^k_\gamma$, are all positive on $I$.  Since $t_1 < \cdots < t_k$, it is easy to check that the variables in the integrand of \eqref{E:Jk} are also ordered:  $s_1 < \cdots < s_{k-1}$, and similarly for the dummy variables used in defining all previous $J_\gamma^j$.  From this and positivity of the $A_\gamma^k$, it follows that the integrand in the definition of $J_\gamma^j$ for $j \leq k$ is non-negative on the domain of integration.

The lemma is trivial when $k=1$.  Let us assume that the estimate for $J_\gamma^{k-1}$ is valid.  Reparametrizing ($t \mapsto -t$) if necessary and applying our almost monotonicity hypothesis if needed, we may assume that 
\begin{equation} \label{E:almost increasing}
\frac{L_\gamma^d (L_\gamma^{d-k})^{k-1}}{(L_\gamma^{d-k+1})^k}(t_1) \leq C \frac{L_\gamma^d (L_\gamma^{d-k})^{k-1}}{(L_\gamma^{d-k+1})^k}(t_2), \:\: \text{for all} \:\: t_1 \leq t_2.
\end{equation}
By the induction hypothesis and the positivity remarked above,
\begin{align*}
&J_\gamma^k(t_1,\ldots,t_k) = \prod_{j=1}^k A_\gamma^k(t_j) \int_{t_1}^{t_2} \cdots \int_{t_{k-1}}^{t_k} J_\gamma^{k-1}(s_1,\ldots,s_{k-1}) \, ds_{k-1} \cdots ds_1 \\
&\quad \gtrsim \prod_{j=1}^k A_\gamma^k(t_j) \int_{t_1}^{t_2} \cdots \int_{t_{k-1}}^{t_k} \prod_{i=1}^{k-1}\prod_{j=1}^{k-1} A_\gamma^j(s_i)^{\frac j{k-1}} \prod_{1 \leq i < j \leq k-1} (s_j-s_i) \, ds_{k-1} \cdots ds_1.
\end{align*}

Let $B_\gamma^{k-1}(s) := \prod_{j=1}^{k-1} A_\gamma^j(s)^{\frac j{k-1}}$.  By the computation in the previous paragraph and positivity of the integrand,
\begin{align} \label{E:Jgammak B}
&J_\gamma^k(t_1,\ldots,t_k)\\\notag
&\quad  \gtrsim \prod_{j=1}^k A_\gamma^k(t_j) \int_{t_1'}^{t_2''}\int_{t_2'}^{t_3''} \cdots \int_{t_{k-1}'}^{t_k''} \prod_{i=1}^{k-1}B_\gamma^{k-1}(s_i) \prod_{1 \leq i < j \leq k-1} (s_j-s_i) \, ds_{k-1} \cdots ds_1,
\end{align}
where 
$$
t_j' = \tfrac{k-j}k t_j + \tfrac jk t_{j+1}, \quad 1 \leq j \leq k-1, \qquad t_j'' = \tfrac12 (t_{j-1}'+t_j), \qquad 2 \leq j \leq k.
$$

By \eqref{E:A to log guy} and our parametrization, we know that $B_\gamma^{k-1}$ is almost increasing and almost log-concave.  Using these facts, it can be shown (see \cite[Proof of Lemma 2.1]{OberlinJFA})  that
$$
B_\gamma^{k-1}(\theta s_1 + (1-\theta)s_2) \gtrsim B_\gamma^{k-1}(s_1)^\theta B_\gamma^{k-1}(s_2)^{1-\theta}, \quad s_1,s_2 \in I,
$$
for all $\theta$ of the form $\frac{k-j}{k}$, $1\le j\le k-1$. Therefore on the domain of integration in \eqref{E:Jgammak B},
$$
\prod_{i=1}^{k-1} B_\gamma^{k-1}(s_i) \gtrsim \prod_{i=1}^{k-1}B_\gamma^{k-1}(t_i') \gtrsim \prod_{i=1}^{k-1} B_\gamma^{k-1}(t_i)^{\frac{k-i}k} B_\gamma^{k-1}(t_{i+1})^{\frac ik} = \prod_{i=1}^k B_\gamma^{k-1}(t_i)^{\frac{k-1}k}
$$
and if $1 \leq i < j \leq k-1$ (since $t_1 \leq t_2 \leq \cdots \leq t_k$),
\begin{align*}
s_j-s_i &= (s_j-t_j) + (t_j-t_{i+1}) + (t_{i+1}-s_i) \\
&\gtrsim (t_{j+1}-t_j) + (t_j-t_{i+1}) + (t_{i+1}-t_i) = t_{j+1}-t_i.
\end{align*}
Thus
\begin{align*}
J_\gamma^k(t_1,\ldots,t_k) &\gtrsim \prod_{j=1}^k A_\gamma^k(t_j) B_\gamma^{k-1}(t_j)^{\frac{k-1}k} \prod_{1 \leq i < j-1 \leq k-1} (t_j-t_i) \prod_{2 \leq j \leq k} (t_j-t_{j-1}) \\
&= \prod_{i=1}^k \prod_{j=1}^k A_\gamma^j(t_i)^{\frac jk} \prod_{1 \leq i < j \leq k} (t_j-t_i).
\end{align*}
This completes the proof of the lemma and thereby Proposition~\ref{P:convl gi}.
\end{proof}

Now we turn to Proposition~\ref{P:jacobian bounds}.  Recall that we assume that $\gamma:I \to \R^d$ is a $C^d$ curve and that the geometric inequality relating $J_\gamma$ and $L_\gamma$ holds.  We want to prove that the geometric inequalities (\ref{E:Psi even}-\ref{E:Psi odd}) hold.  These estimates were proved in the polynomial case in \cite{DSjfa}, but certain aspects of that proof do not readily generalize.  We give the details for \eqref{E:Psi even} only, which we restate for the convenience of the reader.  Assume that $d+1=2D$ and recall that 
$$
\Psi^{2D}_{(t_0,y_0)}(s_1,t_1,\dots,s_K,t_K)  = \bigl(t_K,y_0 + \sum_{j=1}^Ks_j(\gamma(t_{j-1}) - \gamma(t_j))\bigr), 
$$
and that \eqref{E:Psi even} is the inequality
\begin{align*}
&|\det\bigl(D\Psi^{d+1}_{(t_0,y_0)}(s_1,t_1,\dots,s_D,t_D)\bigr)| \\ \notag
&\qquad  \gtrsim \prod_{i=1}^{D-1}\bigl\{|s_{i+1}-s_i||L_\gamma(t_i)|^{\frac2{d+1}}\prod_{\stackrel{0 \leq j \leq D}{j\neq i}}|t_j-t_i|^2\bigr\} |L_\gamma(t_0)|^{\frac1{d+1}}|L_\gamma(t_D)|^{\frac1{d+1}}|t_D-t_0|.
\end{align*}

\begin{proof}[Proof of \eqref{E:Psi even} of Proposition~\ref{P:jacobian bounds}]
The identity 
\begin{align} \label{E:DPsi even}
& \det\bigl(D\Psi^{d+1}_{(t_0,y_0)}(t_1,s_1,\dots,t_D,s_D)\bigr) = \\\notag
&\qquad \pm \bigl\{\prod_{i=1}^{D-1}(s_{i+1}-s_i)\bigr\}\bigl\{\prod_{j=D+1}^{2D-1}\partial_j|_{t_j = t_{j-D}}\bigr\} \det\left(\begin{array}{ccc} 1 & \cdots & 1 \\ \gamma(t_0) & \cdots & \gamma(t_{2D-1}) \end{array} \right)
\end{align}
may be proved using Gaussian elimination; for details, see Lemma~4.3 of \cite{DSjfa}.  We will ignore the $s_j$ for the remainder of the argument.  

Using elementary matrix manipulations, we may write
\begin{align} \notag
&\det\left(\begin{array}{ccc} 1 & \cdots & 1\\ \gamma(t_0) & \cdots & \gamma(t_{2D-1})\end{array}\right)
= \prod_{j=D+1}^{2D-1} (t_j-t_{j-D})\\\notag
&\qquad\qquad \times \det \left(\begin{array}{cccccc} 1 & \cdots & 1& 0 & \cdots &0\\ \gamma(t_0) & \cdots & \gamma(t_0) & \delta(t_1,t_{D+1}) & \cdots & \delta(t_{D-1},t_{2D-1}) \end{array}\right)\\\label{E:rewrite 1gamma}
&\qquad =: \prod_{j=D+1}^{2D-1}(t_j-t_{j-D}) F(t_0,\ldots,t_{2D-1}),
\end{align}
where
$$
\delta(s,t) := \begin{cases} \tfrac1{t-s}(\gamma(t)-\gamma(s)), \quad &t\neq s\\ \gamma'(s), \quad &\text{t=s}.\end{cases}
$$
Since $d \geq 2$ and $\gamma \in C^d$, $\delta \in C^1$, which implies that $F\in C^1$ as well.  

We use the product rule and apply the derivatives in \eqref{E:DPsi even} to the right side of \eqref{E:rewrite 1gamma}, but the contribution from any term in which the derivative falls on $F$ is zero.  Thus
$$
\bigl\{\prod_{j=D+1}^{2D-1}\partial_j|_{t_j = t_{j-D}}\bigr\} \det\left(\begin{array}{ccc} 1 & \cdots & 1 \\ \gamma(t_0) & \cdots & \gamma(t_{2D-1}) \end{array} \right) = F(t_0,\ldots,t_D,t_1,\ldots,t_{D-1}).
$$
On the set where $t_j \neq t_{j-D}$, for all $D+1 \leq j \leq 2D-1$, by our assumption that \eqref{E:convl GI} holds,
$$
|F(t_0,\ldots,t_{2D-1})| \gtrsim \prod_{i=0}^{2D-1} |L_\gamma(t_j)|^{\frac1{2D}} \frac{\prod_{0 \leq i < j \leq 2D-1}|t_j-t_i|}{\prod_{j=D+1}^{2D-1}|t_j-t_{j-D}|},
$$
so \eqref{E:Psi even} follows from the continuity of both sides of this inequality and a careful accounting of the $t_i$.  
\end{proof}

Next we prove Proposition~\ref{P:one to one}, which asserts the near-injectivity of the iteration maps for the restricted X-ray transform under the hypothesis that $J_\gamma(t_1,\ldots,t_d)$ is nonzero whenever the $t_i$ are distinct.

\begin{proof}[Proof of Proposition~\ref{P:one to one}]  The argument is very much inspired by an argument in \cite{DrM2}.  We will give the details for $\Phi^{d+1}_{(s_0,x_0)}$ in the case $d+1=2D$.  Recalling \eqref{E:Phi even} and fixing $(\tau,\xi) \in \R^{1+d}$, we estimate
\begin{equation} \label{E:card tilde}
\# \{(t,s) \in \Delta : \Phi^{d+1}_{(s_0,x_0)}(t,x) = (\tau,\xi)\} \leq \#\{(t,u)\in \tilde\Delta : \tilde\Phi(t,u) = (\tau,\xi)\},
\end{equation}
where 
\begin{gather*}
\tilde\Phi(t,u) := \sum_{j=1}^D u_j\gamma(t_j), \\
 \tilde\Delta = \{(t,u) \in I^D\times \R^D : t_i \neq t_j \: \forall \: i\neq j, \: u_i \neq 0\:  \forall\: i, \: \sum_{i=1}^D u_i = \tau-s_0\}.
\end{gather*}

If the cardinality of the right side of \eqref{E:card tilde} is greater than $D!$, there exist distinct points $(t,u), (t',u') \in I^D\times \R^D$ such that $\tilde\Phi(t,u) = \tilde\Phi(t',u')$, $t_1 < \cdots < t_D$, $t_1' < \cdots < t_D'$, $u_i,u_i' \neq 0$ for all $i$, and $\sum u_i = \sum u_i'$.  By collecting like terms, we may rewrite the equation $\tilde\Phi(t,u) = \tilde\Phi(t',u')$ as
$$
\sum_{j=1}^\ell v_j\tilde\Phi(w_j) = 0,
$$
for some $2 \leq \ell \leq 2D = d+1$,  $\{w_1,\ldots,w_\ell\} \subset \{t_1,\ldots,t_D,t_1',\ldots,t_D'\}$ with $w_1 < \cdots < w_\ell$, and nonzero $v_j$ satisfying $\sum_j v_j = 0$.  

By inserting extra points between $w_1$ and $w_2$, allowing some of the $v_j$ to be zero, and relabeling if needed, 
$$
\sum_{j=1}^{d+1} v_j \tilde\Phi(w_j) = 0,
$$
for some $w_1 < \cdots < w_d$ in $I$ and $v_j \in \R$, not all zero, with $\sum_j v_j = 0$.  Let $\alpha_k = \sum_{j=1}^k v_j$, $1 \leq j \leq d$; then
$$
\sum_{j=1}^d \alpha_j (\gamma(w_j)-\gamma(w_{j+1})) = 0.
$$
By our assumption on the $v_j$, the $\alpha_j$ are not all zero, so 
$$
\{\gamma(w_j)-\gamma(w_{j+1})\}_{j=1}^d
$$
is linearly dependent.  Combining this with the fundamental theorem of calculus and multilinearity of the determinant, 
\begin{align*}
0 &= \det\Big(\int_{w_1}^{w_2} \gamma'(t_1)\, dt_1, \ldots, \int_{w_d}^{w_{d-1}} \gamma'(t_d)\, dt_d\Big)\\
&= \int_{w_1}^{w_2} \cdots \int_{w_d}^{w_{d+1}} \det(\gamma'(t_1),\ldots,\gamma'(t_d))\, dt_1\, \cdots \, dt_d.
\end{align*}
But since $t_1 < \cdots < t_d$ on the domain of integration, our assumption on $J_\gamma$ implies that the right side is nonzero, a contradiction.  Tracing back, this implies that $\Phi^{d+1}_{(s_0,x_0)}$ has the claimed almost-injectivity.  

We leave the details in the remaining cases to the interested reader.  
\end{proof}

Finally, we prove Proposition~\ref{P:convl exp gi}, which asserts that the estimates
\begin{gather*}
| L_\Gamma(t)| \sim e^{-t\sum_{j=1}^d a_i} \prod_{j=1}^d|a_j \Theta_j(\infty)| \prod_{1 \leq i < j \leq d}(a_j-a_i)\\
|J_\Gamma(t_1,\ldots,t_d)| \gtrsim \prod_{i=1}^d |L_\Gamma(t_i)|^{\frac1d}\prod_{1 \leq i < j \leq d} (|t_i-t_j|e^{c_d(a_d-a_1)|t_i-t_j|}).
\end{gather*}
hold on $[\tau,\infty)$ and $[\tau,\infty)^d$, respectively, for some sufficiently large $\tau>0$, provided
$$
\Gamma(t) = (e^{-a_1t}\Theta_1(t),\ldots,e^{-a_dt}\Theta_d(t))
$$
for nonzero real numbers $a_1 < \cdots < a_d$ and $\Theta_i \in C^d_{\rm{loc}}((\tau,\infty)$ with $\Theta_i(\infty) := \lim_{t \nearrow\infty} \Theta_i(t) \neq 0$ and $\lim_{t \nearrow \infty} \Theta_i^{(m)}(t) = 0$, $1 \leq i, m \leq d$.

\begin{proof}[Proof of Proposition~\ref{P:convl exp gi}]  To avoid repetition from \cite{DMtams}, we merely sketch the details for arguments that already appear there.  Using a continuity argument and \eqref{E:thetas}, we may choose $\tau$ sufficiently large that the $L_\Gamma^k$ do not change sign on $(\tau,\infty)$ and satisfy
$$
| L_\Gamma^k(t)| \sim e^{-t\sum_{j=1}^k a_j}\prod_{j=1}^k |a_j\Theta_j(\infty)| \prod_{1 \leq i < j \leq k}(a_j-a_i), \quad 1 \leq k \leq d.
$$
Applying an affine transformation, we may assume that the $L_\Gamma^k$ are all positive on $(\tau,\infty)$.  Thus the $A_\gamma^k$ satisfy
$$
A_\gamma^k(t) \sim C_{\Gamma,k} \begin{cases} e^{-t(a_{d-k+1}-a_{d-k})}, \quad &1 \leq k \leq d-1,\\ e^{-ta_1}, \quad &k=d,\end{cases}
$$
where 
$$
C_{\Gamma,k} = \begin{cases} \frac{a_{d-k+1}\Theta_{d-k+1}(\infty) \prod_{i=1}^{d-k}(a_{d-k+1}-a_i)}{a_{d-k}\Theta_{d-k}(\infty)\prod_{j=1}^{d-k-1}(a_{d-k}-a_j)}, \quad & 1 \leq k \leq d-1, \\ a_1\Theta_1(\infty), \quad & k=d. \end{cases}
$$

For $1 \leq k \leq d$ and $\tau < t_1 < \cdots < t_k$, define
$$
I^k(t_1,\ldots,t_k) = \begin{cases} 
e^{-t_1(a_d-a_{d-1})}, \quad &k=1\\
e^{-(t_1+\cdots+t_k)(a_{d-k+1}-a_{d-k})}\\
\qquad \times \int_{t_1}^{t_2} \cdots \int_{t_{k-1}}^{t_k}I^{k-1}(s_1,\ldots,s_{k-1})\, ds, \quad & 2 \leq k \leq d-1\\
e^{-(t_1+\cdots+t_d)a_1}\int_{t_1}^{t_2} \cdots \int_{t_{d-1}}^{t_d}I^{d-1}(s_1,\ldots,s_{d-1})\, ds, &k=d.
\end{cases}
$$
Then by \eqref{E:Jk}, if $\tau < t_1 < \cdots < t_d$,
$$
|J_\Gamma^d(t_1,\ldots,t_d)| \sim \bigl(\prod_{j=1}^d C_{\Gamma,j}^j\bigr)I^d(t_1,\ldots,t_d),
$$
and since
$$
\prod_{j=1}^d C_{\Gamma,j}^j = \prod_{j=1}^k a_j \Theta_j(\infty)\prod_{1 \leq i < j \leq k}(a_j-a_i),
$$
to prove \eqref{E:convl exp gi}, it suffices to show that whenever $\tau < t_1 < \cdots < t_d$,
$$
I^d(t_1,\ldots,t_d) \gtrsim \prod_{i=1}^d e^{-\frac1d (t_1+\cdots+t_d)(a_1+\cdots+a_d)}e^{c_d(a_d-a_1)(t_d-t_1)}\prod_{1 \leq i < j \leq d}(t_j-t_i).
$$

This is just the case $k=d$ of the following.
\begin{lemma}\label{L:exp gi lemma}
If $1 \leq k \leq d$, and $\tau < t_1 < \ldots < t_k$,
$$
I^k(t_1,\ldots,t_k) \gtrsim_k 
\begin{cases}
e^{-(\frac1k(a_d+\cdots+a_{d-k+1}) - a_{d-k})(t_1+\cdots+t_k)} \prod_{1 \leq i < j \leq k}(t_j-t_i), \ \ &k< d\\
e^{-\frac1d(a_1+\cdots+a_d)(t_1+\cdots+t_d)}e^{c_d(a_d-a_1)(t_d-t_1)}\\
\qquad \times\prod_{1 \leq i < j \leq d}(t_j-t_i), \quad &k=d.
\end{cases}
$$
\end{lemma}

\begin{proof}[Proof of Lemma~\ref{L:exp gi lemma}]
When $1 \leq k < d$, this is shown in \cite{DMtams} (it also follows from the proof of Proposition~\ref{P:convl gi}), so we only prove the lower bound on $I^d$.  

Let $B = a_d+\cdots+a_2-(d-1)a_1$.  Since the $a_j$ are increasing, $B \geq a_d-a_1$, so by some arithmetic, the lower bound on $I^d$ will follow from
\begin{equation}\label{E:claimed}
\begin{aligned}
\int_{t_1}^{t_2}\cdots\int_{t_{d-1}}^{t_d}& e^{-\frac1{d-1}B(s_1+\cdots+s_{d-1})}\prod_{1 \leq i < j \leq d-1} (s_j-s_i)\, ds\\
&\qquad \gtrsim e^{-\frac1d B(t_1+\cdots+t_d)}e^{c_dB(t_d-t_1)}\prod_{1 \leq i < j \leq d}(t_j-t_i).
\end{aligned}
\end{equation}
For $1 \leq j \leq d-1$, define
$$
m_j = \tfrac{d-j}d t_j + \tfrac jd t_{j+1}, \qquad t_{j+1}' = m_j+\tfrac{(t_{j+1}-t_j)}{2d}, \qquad t_j'' = m_j - \tfrac{(t_{j+1}-t_j)}{2d}.
$$
Then for each $j$, $t_j < t_j'' < t_{j+1}' < t_{j+1}$ and $\tfrac12(t_j''+t_{j+1}') = m_j$.  Therefore 
\begin{align*}
\int_{t_1}^{t_2}\cdots\int_{t_{d-1}}^{t_d} &e^{-\frac1{d-1}B(s_1+\cdots+s_{d-1})}\prod_{1 \leq i < j \leq d-1} (s_j-s_i)\, ds \\
&\qquad \geq \int_{t_1''}^{t_2'}\cdots\int_{t_{d-1}''}^{t_d'} e^{-\frac1{d-1}B(s_1+\cdots+s_{d-1})}\prod_{1 \leq i < j \leq d-1}(s_j-s_i)\, ds\\
&\qquad \sim \prod_{1 \leq i < j \leq d-1}(t_{j+1}-t_i)\prod_{j=2}^d\tfrac{d-1}{B}(e^{-\frac B{d-1}t_{j-1}''}-e^{-\frac B{d-1}t_j'})\\
&\qquad = e^{-\frac B{d-1}(m_1+\cdots+m_{d-1})}\prod_{1 \leq i < j \leq d-1}(t_{j+1}-t_i)\\
&\qquad \qquad \times \prod_{j=2}^d \tfrac{d-1}B(e^{-\frac B{2(d-1)}(t_{j-1}''-t_j')}-e^{-\frac B{2(d-1)}(t_j'-t_{j-1}'')}).
\end{align*}
We turn now to this last term.  It is elementary to show that for any $t \in \R$ and $0 \leq \theta < 1$, $|e^t-e^{-t}| \geq |t|(1-\theta)e^{\theta|t|}$, and furthermore 
$$
(t_d'-t_{d-1}'')+(t_{d-1}'-t_{d-2}'') + \cdots + (t_2'-t_1'') \gtrsim_d t_d-t_1.
$$
Therefore
$$
\prod_{j=2}^d \tfrac{d-1}B(e^{-\frac B{2(d-1)}(t_{j-1}''-t_j')}-e^{-\frac B{2(d-1)}(t_j'-t_{j-1}'')}) \gtrsim_d e^{c_d B(t_d-t_1)}\prod_{j=2}^d (t_j-t_{j-1}).
$$
The claimed estimate \eqref{E:claimed} follows.  
\end{proof}
This completes the proof of Proposition~\ref{P:convl exp gi}.
\end{proof}



\end{document}